\documentclass [11pt,reqno]{amsart}
\usepackage {amsmath,amssymb,verbatim,geometry}
\usepackage[pdftex,hyperindex]{hyperref}

\usepackage{color}

\newcommand{\corr}[1]{{\textcolor{red}{#1}}}

\newcommand{\blue}[1]{{\textcolor{blue}{#1}}}

\newcommand{\C}{\mathbb{C}}
\newcommand{\bbD}{\mathbb{D}}

\newcommand{\Q}{\mathbb{Q}}
\newcommand{\R}{\mathbb{R}}
\newcommand{\Z}{\mathbb{Z}}
\newcommand{\N}{\mathbb{N}}
\renewcommand{\P}{\mathbb{P}}

\newcommand{\tu}{\tilde{u}}

\newcommand{\tcL}{\tilde{\cL}}
\newcommand{\tcX}{\widetilde{\cX}}

\newcommand{\cB}{\mathcal{B}}

\newcommand{\cH}{\mathcal{H}}

\newcommand{\cL}{\mathcal{L}}

\newcommand{\cO}{\mathcal{O}}

\newcommand{\cX}{\mathcal{X}}

\newcommand{\Xdiv}{X^{\mathrm{div}}}

\renewcommand{\a}{\alpha}
\renewcommand{\b}{\beta}
\renewcommand{\d}{\delta}
\newcommand{\e}{\varepsilon}

\newcommand{\g}{\gamma}
\newcommand{\la}{\lambda}
\newcommand{\om}{\omega}
\newcommand{\p}{\psi}

\newcommand{\La}{\Lambda}

\newcommand{\eg}{{\rm e.g.\ }} 
\newcommand{\ie}{{\rm i.e.\ }} 
\newcommand{\triv}{\mathrm{triv}} 
\newcommand{\NA}{\mathrm{NA}}

\newcommand{\An}{\mathrm{An}} 
\newcommand{\amp}{\mathrm{amp}} 
 
\newcommand{\lo}{\mathrm{log}} 
 
\newcommand{\re}{\mathrm{ref}} 
\newcommand{\red}{\mathrm{red}}

\newcommand{\half}{\tfrac12}

\newcommand{\an}{\mathrm{an}}

\DeclareMathOperator{\Spec}{Spec}
\DeclareMathOperator{\Exc}{Exc}

\DeclareMathOperator{\MA}{MA}

\DeclareMathOperator{\DF}{DF}

\DeclareMathOperator{\Pic}{Pic}

\DeclareMathOperator{\ord}{ord}

\DeclareMathOperator{\Aut}{Aut}

\DeclareMathOperator{\reg}{reg}

\DeclareMathOperator{\sing}{sing}

\DeclareMathOperator{\spec}{Spec}

\DeclareMathOperator{\Ric}{Ric}

\DeclareMathOperator{\CH}{CH}

\DeclareMathOperator{\Tr}{Tr}

\DeclareMathOperator{\GL}{GL}
\DeclareMathOperator{\SL}{SL}
\DeclareMathOperator{\U}{U}

\newcommand{\D}{\Delta}

\newcommand{\cro}[1]{[\![#1]\!]}
\newcommand{\lau}[1]{(\!(#1)\!)}

\numberwithin{equation}{section}       

\newtheorem{prop} {Proposition} [section]
\newtheorem{thm}[prop] {Theorem} 
\newtheorem{defi}[prop] {Definition}
\newtheorem{lem}[prop] {Lemma}
\newtheorem{cor}[prop]{Corollary}
\newtheorem{prop-def}[prop]{Proposition-Definition}

\newtheorem*{thmA}{Theorem A}

\newtheorem*{thmC}{Theorem C}

\newtheorem*{corB}{Corollary B} 
 
\newtheorem*{corD}{Corollary D} 
\newtheorem*{corE}{Corollary E} 
\newtheorem{exam}[prop]{Example}
\newtheorem{rmk}[prop]{Remark}
\theoremstyle{remark}

\newtheorem*{ackn}{Acknowledgment} 

\title[Uniform K-stability and asymptotics of energy
functionals]{Uniform K-stability and asymptotics of energy functionals
  in K\"ahler geometry \corr{with errata}}
\date{\today}

\author{S{\'e}bastien Boucksom
  \and
  Tomoyuki Hisamoto
  \and 
  Mattias Jonsson}
\address{CNRS-CMLS\\
  \'Ecole Polytechnique\\
  F-91128 Palaiseau Cedex\\
  France}
\email{sebastien.boucksom@polytechnique.edu}
\address{Graduate School of Mathematics\\
  Nagoya University\\
  Furocho\\
  Chikusa\\
  Nagoya\\ 
  Japan}
\email{hisamoto@math.nagoya-u.ac.jp}
\address{Dept of Mathematics\\
  University of Michigan\\
  Ann Arbor, MI 48109--1043\\
  USA}
\address{Mathematical Sciences\\
  Chalmers University of Technology
  and University of Gothenburg\\
  SE-412 96 G\"oteborg\\
  Sweden}
\email{mattiasj@umich.edu}

\begin{document}

\begin{abstract}
  Consider a polarized complex manifold $(X,L)$ and 
  a ray of positive metrics on $L$ defined by a positive
  metric on a test configuration for $(X,L)$. For many common
  functionals in K\"ahler geometry, we prove that the slope at
  infinity along the ray is given by evaluating the non-Archimedean
  version of the functional (as defined in our earlier
  paper~\cite{BHJ1}) at the non-Archimedean metric on $L$ defined
  by the test configuration.
  Using this asymptotic result, we show that 
  coercivity of the Mabuchi functional implies
  uniform K-stability, as defined in~\cite{Der1,BHJ1}. 
  
\end{abstract}

\maketitle

\setcounter{tocdepth}{1}
\tableofcontents
%
%
%
%
%
%
%
%

\section*{About the errata}
An error was unfortunately found in the published version of this article, which affects a number of results of the paper. The proof of Theorem~\ref{thm:spec} below indeed contains an incorrect claim --- but the statement of this theorem is however possibly true, and we have thus decided in this updated arXiv version to mark \corr{in red} all parts of the paper affected by the issue, which are thus unproved as things stand, but hopefully repairable if the flawed proof every gets corrected. Note that we haven't tried to otherwise update the paper. 

We would like to extend our warmest thanks to Yan Li, from Peking University, for pointing out this problem to us.

\section*{Introduction}
Let $(X,L)$ be a polarized complex manifold, \ie smooth projective complex variety $X$ endowed with an ample line bundle $L$. A central problem in
K\"ahler geometry is to give necessary and sufficient conditions for 
the existence of canonical K\"ahler metrics in the corresponding K\"ahler class $c_1(L)$, for example, 
constant scalar curvature K\"ahler metrics (cscK for short). To fix ideas, suppose the 
reduced automorphism group $\Aut(X,L)/\C^*$ is discrete.
In this case, the celebrated Yau-Tian-Donaldson conjecture asserts that $c_1(L)$
admits a cscK metric iff $(X,L)$ is K-stable.
That K-stability follows from the existence of a cscK metric
was proved by Stoppa~\cite{Sto09}, building upon work by
Donaldson~\cite{Don3}, but the reverse direction is considered wide
open in general. 

This situation has led people to introduce stronger stability
conditions that would hopefully imply the existence of a cscK metric. 
Building upon ideas of Donaldson~\cite{Don3},
Sz\'ekelyhidi~\cite{Sze1} proposed to use a version of K-stability in which,
for any test configuration $(\cX,\cL)$ for $(X,L)$,
the Donaldson-Futaki invariant $\DF(\cX,\cL)$ is bounded below by 
a positive constant times a suitable \emph{norm} of $(\cX,\cL)$.
(See also~\cite{Sze2} for a related notion.)

Following this lead, we defined in the prequel~\cite{BHJ1} to this paper,
$(X,L)$ to be \emph{uniformly K-stable} if there exists $\delta>0$
such that 
\begin{equation*}
  \DF(\cX, \cL) \geq \delta J^{\NA}(\cX, \cL) 
\end{equation*}
for any normal and ample test configuration $(\cX,\cL)$. 
Here $J^{\NA}(\cX, \cL)$ is a non-Archimedean analogue of Aubin's 
$J$-functional. It is equivalent to the $L^1$-norm of $(\cX,\cL)$ as
well as the minimum norm considered by Dervan~\cite{Der1}.
The norm is zero iff the normalization of $(\cX,\cL)$ is trivial, so uniform K-stability implies K-stability.

In~\cite{BHJ1} we advocated the point of view that a test
configuration defines a \emph{non-Archimedean metric} on $L$, 
that is, a metric on the Berkovich analytification of $(X,L)$ with respect to
the trivial norm on the ground field $\C$. 
Further, we defined non-Archimedean analogues 
of many classical functionals in K\"ahler geometry.
One example is the functional $J^\NA$ above. 
Another is $M^\NA$, a non-Archimedean 
analogue of the Mabuchi K-energy functional $M$.
It agrees with the Donaldson-Futaki invariant, up to an explicit error
term, and uniform K-stability is equivalent to 
\begin{equation*}
  M^{\NA}(\cX, \cL) \geq \delta J^{\NA}(\cX, \cL)
\end{equation*}
for any ample test configuration $(\cX,\cL)$.
In~\cite{BHJ1} we proved that canonically polarized manifolds and
polarized Calabi-Yau manifolds are always uniformly K-stable.

\smallskip
A first goal of this paper is to exhibit precise relations between the
non-Archimedean functionals and their classical counterparts. 
From now on we do not \emph{a priori} assume that the reduced
automorphism group of $(X,L)$ is discrete.
We prove
\begin{thmA} 
  Let $(\cX,\cL)$ be an ample test configuration for a polarized complex
  manifold $(X,L)$. 
  Consider any smooth strictly positive $S^1$-invariant 
  metric $\Phi$ on $\cL$ defined near the central fiber, and let $(\phi^s)_s$ 
  be the corresponding ray of smooth positive metrics on $L$. 
  Denoting by $M$ and $J$ the Mabuchi K-energy functional and Aubin
  $J$-functional, respectively, we then have
  \begin{equation*}
    \lim_{s\to+\infty}\frac{M(\phi^s)}{s}=M^{\NA}(\cX,\cL)
    \qquad\text{and}\qquad
    \lim_{s\to+\infty}\frac{J(\phi^s)}{s}=J^{\NA}(\cX,\cL).
  \end{equation*}
\end{thmA} 
The corresponding equalities also hold for several other functionals,
see~Theorem~\ref{thm:asymfunc}. 
More generally, we prove that these asymptotic properties hold
in the logarithmic setting, for subklt pairs $(X,B)$ and with weaker positivity assumptions,
see Theorem~\ref{T201}.

At least when the total space $\cX$ is smooth, the assertion in
Theorem~A regarding the Mabuchi functional is closely related to 
several statements appearing in the literature~\cite[Corollary 2]{PRS},
\cite[Corollary 1]{PT2},
\cite[Remark 12, p.38]{Li12},
\cite[Lemma~2.1]{Tian14},
following the seminal work~\cite{Tian97}. A special case appears already 
in~\cite[p.328]{DT92}.
However, to the best of our knowledge, neither the
general and precise statement given here nor its proof is available in the
literature.

As in~\cite{PRS}, the proof of Theorem~A uses Deligne pairings, but the analysis here is
more delicate since the test configuration $\cX$ is not smooth. Using resolution 
of singularities, we can make $\cX$ smooth, but then we lose the strict positivity 
of $\Phi$. It turns out that the situation can be analyzed by estimating
integrals of the form $\int_{\cX_\tau}e^{2\Psi|_{\cX_\tau}}$ as $\tau\to0$, where $\cX\to\C$ is
an snc test configuration for $X$, and $\Psi$ is a smooth metric
on the (logarithmic) relative canonical bundle of $\cX$ 
near the central fiber, see Lemma~\ref{lem:estim}.

\smallskip
Donaldson~\cite{Don99} (see also~\cite{Mab87,Sem92}) 
has advocated the point of view that the space
$\cH$ of positive metrics on $L$ is an infinite-dimensional 
symmetric space. One can view the space $\cH^\NA$ of positive 
non-Archimedean metrics on $L$ as (a subset of) the associated (conical) Tits building.
Theorem~A gives justification to this paradigm.

\smallskip
The asymptotic formulas in Theorem~A allow us to study coercivity properties
of the Mabuchi functional. 
As an immediate consequence of Theorem~A, we have
\begin{corB}
  If the Mabuchi functional is coercive in the sense that 
  $$
  M \geq \delta J -C 
  $$
  on $\cH$ for some positive constants $\delta$ and $C$, 
  then $(X, L)$ is uniformly K-stable, that is, 
  $$
  \DF(\cX, \cL) \geq \delta J^{\NA}(\cX, \cL)
  $$
  holds for any normal ample test configuration $(\cX, \cL)$. 
\end{corB} 
Coercivity of the Mabuchi functional is known to hold if $X$ is a
K\"{a}hler-Einstein manifold without vector fields. 
This was first established in the Fano case by~\cite{PSSW};
an elegant proof can be found in~\cite{DR}. 
As a special case of a very recent result of Berman, Darvas and
Lu~\cite{BDL16}, coercivity of the Mabuchi functional also holds for 
general polarized varieties admitting a 
metric of constant scalar curvature and having discrete reduced 
automorphism group. 
Thus, if $(X,L)$ admits a constant scalar curvature 
metric and $\Aut(X,L)/\C^*$ is discrete, then $(X,L)$ is 
uniformly K-stable. The converse statement is not currently
known in general, but see below for the Fano case.

\smallskip

Next, we study coercivity of the Mabuchi functional when 
restricted to the space of Bergman metrics. For any $m\ge1$
such that $mL$ is very ample, let $\cH_m$ be the space of 
Fubini-Study type metrics on $L$, induced by the embedding of 
$X\hookrightarrow\P^{N_m}$ via $mL$.

\corr{\begin{thmC}
  Fix $m$ such that $(X,mL)$ is linearly normal, and $\d>0$. Then the
  following conditions are equivalent:
  \begin{itemize}
  \item[(i)]
    there exists $C>0$ such that $M\ge\d J-C$ on $\cH_m$. 
  \item[(ii)] 
    $\DF(\cX_\la,\cL_\la)\ge\d J^{\NA}(\cX_\la,\cL_\la)$ for all 
    1-parameter subgroups $\la$ of $\GL(N_m,\C)$;
  \item[(iii)] 
    $M^{\NA}(\cX_\la,\cL_\la)\ge\d J^{\NA}(\cX_\la,\cL_\la)$ for all 
    1-parameter subgroups $\la$ of $\GL(N_m,\C)$.
  \end{itemize}
  Here $(\cX_\la,\cL_\la)$ is the test configuration for $(X,L)$ defined by $\la$.
\end{thmC}
Note that a different condition equivalent to~(i)--(iii) appears in ~\cite[Theorem~1.1]{Paul13}.}

\corr{The equivalence of~(ii) and~(iii) stems from the close relation
between the Donaldson-Futaki invariant and the non-Archimedean 
Mabuchi functional. In view of Theorem~A, the equivalence between~(i)
and~(iii) can be viewed as a generalization of the Hilbert-Mumford
criterion.
The proof uses in a crucial way a deep result of Paul~\cite{Paul12}, which states that the
restrictions to $\cH_m$ of the Mabuchi functional and the
$J$-functional have log norm singularities (see~\S\ref{sec:CM}).}

\corr{Since every ample test configuration arises as a 1-parameter subgroup 
$\la$ of $\GL(N_m,\C)$ for some $m$, Theorem~C implies
\begin{corD}
  A polarized manifold $(X,L)$ is uniformly K-stable iff there exist
  $\d>0$ and a sequence $C_m>0$ such that $M\ge\d J-C_m$ on $\cH_m$ 
  for all sufficiently divisible $m$. 
\end{corD}
Following Paul and Tian~\cite{PT1,PT2}, we say that $(X,mL)$ is \emph{CM-stable} 
when there exist $C,\d>0$ such that $M\ge\d J-C$ on $\cH_m$.
\begin{corE}
  If $(X,L)$ is uniformly K-stable, then $(X,mL)$ is CM-stable 
  for any sufficiently divisible positive integer $m$. 
  Hence the reduced automorphism group is finite.
\end{corE}
Here the last statement follows from a result by Paul~\cite[Corollary~1.1]{Paul13}.}

\medskip
Let us now comment on the relation of uniform K-stability to the
existence of K\"ahler-Einstein metrics on Fano manifolds.
In~\cite{CDS15}, Chen, Donaldson and Sun proved that a Fano manifold $X$ 
admits a K\"ahler-Einstein metric iff it is K-polystable; see also~\cite{Tian15}.
Since then, several new proofs have appeared. 
Datar and Sz\'ekelyhidi~\cite{DSz15} proved an equivariant
version of the conjecture, using Aubin's original continuity method.
Chen, Sun and Wang~\cite{CSW15} 
gave a proof using the K\"ahler-Ricci flow.

In~\cite{BBJ15}, Berman and the first and last authors of the current
paper used a variational method to prove a slightly different statement:
in the absence of vector fields, the existence of a K\"ahler-Einstein
metric is equivalent to uniform K-stability. In fact, the direct
implication uses Corollary~B above.

\corr{In~\S\ref{S106} we outline a different proof of the fact that a
uniformly K-stable Fano manifold admits a K\"ahler-Einstein metric.
Our method, which largely follows ideas of Tian, 
relies on Sz\'{e}kelyhidi's partial $C^0$-estimates~\cite{Sze3} 
along the Aubin continuity path, together with Corollary~D.}

\smallskip
As noted above, uniform K-stability implies that the reduced
automorphism group of $(X,L)$ is discrete. In the presence of vector
fields, there should presumably be a natural notion of uniform K-polystability. 
We hope to address this in future work.

\smallskip
There have been several important developments since a first draft of
the current paper was circulated.
First, Z.~Sj\"ostr\"om Dyrefelt~\cite{SD16} and, independently,
R.~Dervan and J.~Ross~\cite{DR16}, proved a transcendental version of Theorem~A.
Second, as mentioned above, it was proved in~\cite{BBJ15} that in the 
case of a Fano manifold without holomorphic vector fields, 
uniform K-stability is equivalent to coercivity of the Mabuchi
functional, and hence to the existence of a K\"ahler-Einstein metric.
Finally, the results in this paper were used in~\cite{BDL16} to
prove that an arbitrary polarized pair $(X,L)$ admitting a cscK metric
must be K-polystable.

\medskip
The organization of the paper is as follows. 
In the first section, we review several classical energy functionals in
K\"ahler geometry and their interpretation as metrics on suitable Deligne pairings. 
Then, in~\S\ref{S104}, we recall some non-Archimedean notions
from~\cite{BHJ1}.
Specifically, a non-Archimedean metric is an equivalence
class of test configurations, and the 
non-Archimedean analogues of the energy functionals in~\S\ref{S103}
are defined using intersection numbers.
In~\S\ref{S105} we prove Theorem~A relating the classical and 
non-Archimedean functionals via subgeodesic rays.
These results are generalized to the logarithmic setting in~\S\ref{S110}.
Section~\ref{sec:CM} is devoted to the relation between uniform
K-stability and CM-stability. In particular, we prove Theorem~C and 
Corollaries~D and~E.
Finally, in~\S\ref{S106}, we show how to use Sz\'ekelyhidi's partial
$C^0$-estimates along the Aubin continuity path together with
CM-stability to prove that a uniformly K-stable Fano manifold admits
a K\"ahler-Einstein metric.

\begin{ackn} 
The authors would like to thank Robert Berman for very useful
discussions. The first author is also grateful to Marco Maculan,
Vincent Guedj and Ahmed Zeriahi for helpful conversations. He was
partially supported by the ANR projects GRACK, MACK and POSITIVE.\@
The second author was supported by JSPS KAKENHI Grant Number 25-6660 and 15H06262.   
The last author was partially supported by NSF grant DMS-1266207,
the Knut and Alice Wallenberg foundation, 
and the United States---Israel Binational Science Foundation.

\end{ackn}
%
%
%
%
\section{Deligne pairings and energy functionals}\label{S103}
In this section we recall the definition and main properties of the
Deligne pairing, as well as its relation to classical functionals in 
K\"ahler geometry.
%
%
\subsection{Metrics on line bundles}\label{S107}
We use additive notation for line bundles and metrics.
If, for $i=1,2$, $\phi_i$ is a metric on a line bundle $L_i$ on
$X$ and $a_i\in\Z$, then $a_1\phi_1+a_2\phi_2$ is a metric on $a_1L_1+a_2L_2$.
This allows us to define metrics on $\Q$-line bundles.
A metric on the trivial line bundle will be identified with a function on $X$. 

If $\sigma$ is a (holomorphic) section of a line bundle $L$ on a complex analytic space $X$, 
then $\log|\sigma|$ stands for the corresponding (possibly singular) metric on $L$.
For any metric $\phi$ on $L$, 
$\log|\sigma|-\phi$ is therefore a function, and
\begin{equation*}
  |\sigma|_\phi:=|\sigma|e^{-\phi}=\exp(\log|\sigma|-\phi)
\end{equation*}
is the length of $\sigma$ in the metric $\phi$. 

We normalize the operator $d^c$ so that $dd^c=\tfrac{i}{\pi}\partial\bar\partial$, and set (somewhat abusively) 
\begin{equation*}
  dd^c\phi:=-dd^c\log|\sigma|_\phi
\end{equation*}
for any local trivializing section $\sigma$ of $L$. The globally defined $(1,1)$-form (or current) 
$dd^c\phi$ is the curvature of $\phi$, normalized so that it represents the (integral) first Chern class of $L$. 

If $X$ is a complex manifold of dimension $n$ and $\eta$ is a holomorphic $n$-form
on $X$, then 
\begin{equation*}
  |\eta|^2:=\frac{i^{n^2}}{2^n}\eta\wedge\bar\eta
\end{equation*}
defines a natural (smooth, positive) volume form on $X$. 
More generally, there is a bijection between smooth metrics on the canonical bundle $K_X$ 
and (smooth, positive) volume forms on $X$, which associates to a smooth metric $\phi$ 
on $K_X$ the volume form $e^{2\phi}$ locally defined by 
\begin{equation*}
  e^{2\phi}:=|\eta|^2/|\eta|^2_\phi
\end{equation*}
for any local section $\eta$ of $K_X$. 

If $\omega$ is a positive $(1,1)$-form on $X$ and $n=\dim X$, 
then $\omega^n$ is a volume form, so $-\half\log\omega^n$ is a 
metric on $-K_X$ in our notation. The \emph{Ricci form} of $\om$ is defined as the curvature
\begin{equation*}
  \Ric\omega:=-dd^c\half\log\omega^n
\end{equation*}
of $\omega$ of this metric; 
it is thus a smooth $(1,1)$-form in the cohomology class $c_1(X)$ of $-K_X$. 

If $\phi$ is a smooth positive metric on a line bundle $L$ on $X$, we denote by $S_\phi\in C^\infty(X)$ 
the \emph{scalar curvature} of the K\"ahler form $dd^c\phi$; it satisfies 
\begin{equation}\label{equ:scal}
  S_\phi(dd^c\phi)^n=n\Ric(dd^c\phi)\wedge(dd^c\phi)^{n-1}.
\end{equation}

%
%
\subsection{Deligne pairings}
While the construction below works in greater generality
\cite{Elk1,Zha96,MG}, we will restrict ourselves to the following
setting. Let $\pi\colon Y\to T$ be a flat, projective morphism between smooth complex 
algebraic varieties, of relative dimension $n\ge0$. Given line bundles 
$L_0,\dots,L_n$ on $Y$, consider the intersection product 
$$
L_0\cdot\ldots\cdot L_n\cdot[Y]\in\CH_{\dim Y-(n+1)}(Y)=\CH_{\dim T-1}(Y). 
$$
Its push-forward belongs to $\CH_{\dim T-1}(T)=\Pic(T)$ since $T$ is
smooth, and hence defines an \emph{isomorphism class} of line bundle on $T$. The \emph{Deligne pairing} of $L_0,\dots,L_n$ selects in a canonical way a specific representative of this isomorphism class, denoted by
$$
\langle L_0,\dots,L_n\rangle_{Y/T}.
$$ 
The pairing is functorial, multilinear, and commutes with base change. It further satisfies the following key inductive property: if $Z_0$ is a non-singular divisor in $Y$, flat over $T$ and defined by a section $\sigma_0\in H^0(Y,L_0)$, then we have a canonical identification
\begin{equation}\label{equ:isomind}
\langle L_0,\dots,L_n\rangle_{Y/T}=\langle L_1|_{Z_0},\dots,L_n|_{Z_0}\rangle_{Z_0/T}.
\end{equation}
For $n=0$, $\langle L_0\rangle_{Y/T}$ coincides with the norm of $L_0$
with respect to the finite flat morphism $Y\to T$. These properties
uniquely characterize the Deligne pairing. Indeed, writing each $L_i$ as a difference 
of very ample line bundles, multilinearity reduces the situation to the case where the $L_i$ 
are very ample. We may thus find non-singular divisors $Z_i\in|L_i|$ with $\bigcap_{i\in I} Z_i$ 
non-singular and flat over $T$ for each set $I$ of indices, and we get
$$
\langle L_0,\dots,L_n\rangle_{Y/T}=\langle L_n|_{Z_0\cap\dots\cap Z_{n-1}}\rangle_{Z_0\cap\dots\cap Z_{n-1}/T}.
$$
%
%
\subsection{Metrics on Deligne pairings}
We use~\cite{Elk2,Zha96,Mor} as references. Given a smooth metric $\phi_j$ on each $L_j$, the Deligne pairing $\langle L_0,\dots,L_n\rangle_{Y/T}$ can be endowed with a continuous metric 
$$
\langle\phi_0,\dots,\phi_n\rangle_{Y/T},
$$
smooth over the smooth locus of $\pi$, the construction being functorial, multilinear, 
and commuting with base change. It is basically constructed by requiring that 
\begin{equation}\label{equ:metrind}
  \langle\phi_0,\dots,\phi_n\rangle_{Y/T}
  =\langle\phi_1|_{Z_0},\dots,\phi_n|_{Z_0}\rangle_{Z_0/T}
  -\int_{Y/T}\log|\sigma_0|_{\phi_0}dd^c\phi_1\wedge\dots\wedge dd^c\phi_n
\end{equation}
in the notation of~\eqref{equ:isomind}, with $\int_{Y/T}$ denoting
fiber integration, \ie the push-forward by $\pi$ as a current. By
induction, the continuity of the metric $\langle\phi_0,\dots,\phi_n\rangle$ 
reduces to that of $\int_{Y/T}\log|\sigma_0|_{\phi_0}dd^c\phi_1\wedge\dots\wedge dd^c\phi_n
$, and thus follows from~\cite[Theorem 4.9]{Stol}. 

\begin{rmk} 
  As explained in~\cite[I.1]{Elk2}, arguing by induction, the key point in checking 
  that~\eqref{equ:metrind} is well-defined is the following symmetry property: 
  if $\sigma_1\in H^0(Y, L_1)$ is a section with divisor $Z_1$ such that both $Z_1$ and $Z_0\cap Z_1$ 
  are non-singular and flat over $T$, then 
  \begin{align*}
    &\int_{Y/T}\log|\sigma_0|_{\phi_0}dd^c\phi_1\wedge\a+\int_{Z_0/T}\log|\sigma_1|_{\phi_1}\a\\
    =&\int_{Y/T}\log|\sigma_1|_{\phi_1}dd^c\phi_0\wedge\a+\int_{Z_1/T}\log|\sigma_0|_{\phi_0}\a
  \end{align*}
  with $\a=dd^c\phi_2\wedge\dots\wedge dd^c\phi_n$. 
  By the Lelong--Poincar\'e formula, the above equality reduces to
  $$
  \pi_*\left(\log|\sigma_0|_{\phi_0}dd^c\log|\sigma_1|_{\phi_1}\wedge\a\right)
  =\pi_*\left(\log|\sigma_1|_{\phi_1}dd^c\log|\sigma_0|_{\phi_0}\wedge\a\right), 
  $$
  which holds by Stokes' formula applied to a monotone regularization of 
  the quasi-psh functions $\log|\sigma_i|_{\phi_i}$. 
\end{rmk}
Metrics on Deligne pairings satisfy the following two crucial properties, which are direct consequences of~\eqref{equ:metrind}. 
\begin{itemize}
\item[(i)] The curvature current of $\langle\phi_0,\dots,\phi_n\rangle_{Y/T}$ satisfies  
\begin{equation}\label{equ:curv}
dd^c\langle\phi_0,\dots,\phi_n\rangle_{Y/T}=\int_{Y/T}dd^c\phi_0\wedge\dots\wedge dd^c\phi_n, 
\end{equation}
where again $\int_{Y/T}$ denotes fiber integration.
\item[(ii)] Given another smooth metric $\phi_0'$ on $L_0$, we have the change of metric formula
\begin{equation}\label{equ:change} 
\langle\phi'_0,\phi_1,\dots,\phi_n\rangle_{Y/T}-\langle\phi_0,\phi_1,\dots,\phi_n\rangle_{Y/T}=\int_{Y/T}(\phi'_0-\phi_0)dd^c\phi_1\wedge\dots\wedge dd^c\phi_n. 
\end{equation}
\end{itemize}
%
%
\subsection{Energy functionals}\label{S101}
Let $(X,L)$ be a polarized manifold, \ie a smooth projective complex variety $X$ with an ample 
line bundle $L$. Set
\begin{equation*}
  V:=(L^n)
  \quad\text{and}\quad
  \bar S:=-nV^{-1}(K_X\cdot L^{n-1}),
\end{equation*}
where $n=\dim X$. 
Denote by $\cH$ the set of smooth positive metrics $\phi$ on $L$. 
For $\phi\in\cH$, set $\MA(\phi):=V^{-1}(dd^c\phi)^n$.
Then $\MA(\phi)$ is a probability measure equivalent to Lebesgue measure, and 
$\int_X S_\phi\MA(\phi)=\bar S$ by~\eqref{equ:scal}. 

We recall the following functionals in K\"ahler geometry. Fix a
reference metric $\phi_{\re}\in\cH$. Our notation largely follows~\cite{BBGZ,BBEGZ}.
\begin{itemize}
\item[(i)]
  The \emph{Monge-Amp\`ere energy functional} is given by
  \begin{equation}\label{equ:E}
    E(\phi)
    =\frac{1}{n+1}\sum_{j=0}^nV^{-1}\int_X(\phi-\phi_{\re})(dd^c\phi)^j\wedge (dd^c\phi_{\re})^{n-j}.
  \end{equation}
\item[(ii)]
  The \emph{$J$-functional} is a translation invariant version of $E$, defined as
  \begin{equation}\label{equ:J}
    J(\phi):=\int_X(\phi-\phi_{\re})\MA(\phi_{\re})-E(\phi).  
  \end{equation}
  The closely related \emph{$I$-functional} is defined by 
  \begin{equation}\label{equ:I}
    I(\phi):=\int_X(\phi-\phi_{\re})\MA(\phi_{\re})-\int_X(\phi-\phi_{\re})\MA(\phi).
  \end{equation}
\item[(iii)]
  For any closed $(1,1)$-form $\theta$, the \emph{$\theta$-twisted
    Monge-Amp\`ere energy} is given by
  \begin{equation}\label{equ:ERic}
    E_\theta(\phi)
    =\frac1n\sum_{j=0}^{n-1}V^{-1}\int_X(\phi-\phi_{\re})(dd^c\phi)^j\wedge (dd^c\phi_{\re})^{n-1-j}\wedge\theta. 
  \end{equation}
  Taking $\theta:=-n\Ric(dd^c \phi_{\re})$, we obtain the 
  \emph{Ricci energy}  $R:=-E_{n\Ric(dd^c\phi_{\re})}$. 
\item[(iv)]
  The \emph{entropy} of $\phi\in\cH$ is defined as 
  \begin{equation}\label{equ:H}
    H(\phi):=\half\int_X\log\left[\frac{\MA(\phi)}{\MA(\phi_{\re})}\right]\MA(\phi),
  \end{equation}
  that is, (half) the relative entropy of the probability measure
  $\MA(\phi)$ with respect to $\MA(\phi_{\re})$. 
  We have $H(\phi)\ge 0$, with equality iff $\phi-\phi_{\re}$ is constant. 
\item[(v)] 
  The \emph{Mabuchi functional} (or K-energy) can now be defined via
  the Chen-Tian formula~\cite{Che2} (see also~\cite[Proposition
  3.1]{BB}) as 
  \begin{equation}\label{equ:varM}
    M(\phi)=H(\phi) +R(\phi)+\bar S E(\phi).
  \end{equation}
\end{itemize}

These functionals vanish at
$\phi_\re$ and satisfy the variational formulas:
\begin{align*}
  \delta E(\phi) &=\MA(\phi)=V^{-1}(dd^c\phi)^n\\
  \delta E_\theta(\phi) &= V^{-1}(dd^c\phi)^{n-1}\wedge\theta\\
  \delta R(\phi) &=-nV^{-1}(dd^c\phi)^{n-1}\wedge\Ric(dd^c\phi_\re)\\
  \delta H(\phi) &=nV^{-1}(dd^c\phi)^{n-1}\wedge(\Ric(dd^c\phi_\re)-\Ric(dd^c\phi))\\
  \delta M(\phi) &=(\bar S-S_\phi)\MA(\phi)
\end{align*}
In particular, $\phi$ is a critical point of $M$ iff $dd^c\phi$ is a cscK metric. 

The functionals $I$, $J$ and $I-J$ are comparable in the sense that 
\begin{equation}\label{e106}
  \frac1n J\le I-J\le nJ
\end{equation}
on $\cH$.
For $\phi\in\cH$ we have 
$J(\phi)\ge 0$, with equality iff $\phi-\phi_{\re}$ is
constant. These properties are thus also shared by $I$ and $I-J$.

The functionals $H$, $I$, $J$, $M$ are translation invariant in the
sense that $H(\phi+c)=H(\phi)$ for $c\in\R$. For $E$ and $R$ we instead have
$E(\phi+c)=E(\phi)+c$ and $R(\phi+c)=R(\phi)-\bar Sc$, respectively.
%
%
\subsection{Energy functionals as Deligne pairings}\label{S108}
The functionals above can be expressed using Deligne pairings, 
an observation going back at least to~\cite{PS6}. Note that any metric $\phi\in\cH$ induces a 
smooth metric $\half\log\MA(\phi)$ on $K_X$. 
The following identities are now easy consequences of the change of 
metric formula~\eqref{equ:change}.
\begin{lem}\label{lem:funcdel} 
  For any $\phi\in\cH$ we have
  \begin{align*}
    (n+1)VE(\phi)
    &=\langle\phi^{n+1}\rangle_X
      -\langle\phi_{\re}^{n+1}\rangle_X;\\
    VJ(\phi)
    &=\langle\phi,\phi_{\re}^n\rangle_X
      -\langle\phi_{\re}^{n+1}\rangle_X
      -\frac{1}{n+1}\left[\langle\phi^{n+1}\rangle_X 
      -\langle\phi_{\re}^{n+1}\rangle_X\right];\\
    VI(\phi)
    &=\langle\phi-\phi_\re,\phi_{\re}^n\rangle_X
      -\langle\phi-\phi_\re,\phi^n\rangle_X;\\
    V R(\phi)
    &=\langle\half\log\MA(\phi_{\re}),\phi^n\rangle_X
      -\langle\half\log\MA(\phi_{\re}),\phi_{\re}^n\rangle_X;\\
    V H(\phi) 
    &=\langle\half\log\MA(\phi),\phi^n\rangle_X 
      -\langle\half\log\MA(\phi_{\re}),\phi^n\rangle_X;\\
    V M(\phi) 
    &=\langle\half\log\MA(\phi),\phi^n\rangle_X 
      -\langle\half\log\MA(\phi_{\re}),\phi_\re^n\rangle_X\\
    &+\frac{\bar S}{n+1}\left[
      \langle\phi^{n+1}\rangle_X-\langle\phi_{\re}^{n+1}\rangle_X
      \right],
  \end{align*}
  where $\langle\ \rangle_X$ denotes the Deligne pairing with respect
  to the constant map $X\to\{\mathrm{pt}\}$.
\end{lem}
\begin{rmk}
  The formulas above make it evident that instead of fixing a
  reference metric $\phi_\re\in\cH$, we could view 
  $E$, $H+R$ and $M$ as metrics on suitable multiples of the complex lines 
  $\langle L^{n+1}\rangle_X$, 
  $\langle K_X,L^n\rangle_X$,
  and $(n+1)\langle K_X,L^n\rangle_X+\bar S \langle L^{n+1}\rangle_X$,
  respectively.
\end{rmk}
\begin{rmk}\label{R201}
  In the definition of $R$, we could replace $-\Ric dd^c\phi_\re$ by 
  $dd^c\p_\re$ for any smooth metric $\p_\re$ on $K_X$.
  Similarly, in the definition of $H$, we could replace the reference measure $\MA(\phi_\re)$
  by $e^{2\p_\re}$. Doing so, and keeping the Chen-Tian formula, 
  would only change the Mabuchi functional $M$ by an additive constant. 
\end{rmk}
%
%
\subsection{The Ding functional}
Now suppose $X$ is a Fano manifold, that is, $L:=-K_X$ is ample. Any
metric $\phi$ on $L$ then induces a positive volume form $e^{-2\phi}$ on
$X$. The \emph{Ding functional}~\cite{Din88} on $\cH$ is defined by 
\begin{equation*}
  D(\phi)=L(\phi)-E(\phi),
\end{equation*}
where
\begin{equation*}
  L(\phi)=-\half\log\int_Xe^{-2\phi}.
\end{equation*}
This functional has proven an extremely useful tool for the study of the existence of 
K\"ahler-Einstein metrics, which are realized as the critical points of $D$, see~\eg~\cite{Berm16,BBJ15}. 
%
%
%
%
\section{Test configurations as non-Archimedean metrics}\label{S104}
In this section we recall some notions and results from~\cite{BHJ1}.  
Let $X$ be a smooth projective complex variety and $L$ a
line bundle on $X$.
%
%
\subsection{Test configurations}
As in~\cite{BHJ1} we adopt the following flexible terminology for test configurations.
\begin{defi}\label{defi:test} A test configuration $\cX$ for $X$ consists of the following data: 
\begin{itemize}
\item[(i)] a flat, projective morphism of schemes $\pi\colon\cX\to\C$; 
\item[(ii)] a $\C^*$-action on $\cX$ lifting the canonical action on $\C$;  
\item[(iii)] an isomorphism $\cX_1\simeq X$. 
\end{itemize}
\end{defi}
We denote by $\tau$ the coordinate on $\C$, and by $\cX_\tau$ the fiber over $\tau$. 

These conditions imply that $\cX$ is reduced and irreducible~\cite[Proposition~2.6]{BHJ1}). 
If $\cX,\cX'$ are test configurations for $X$, then there is a unique
$\C^*$-equivariant birational map $\cX'\dashrightarrow\cX$ compatible
with the isomorphism in~(iii). We say that $\cX'$ \emph{dominates}
$\cX$ if this birational map is a morphism; when it is an isomorphism
we somewhat abusively identify $\cX$ and $\cX'$. Any test
configuration $\cX$ is dominated by its \emph{normalization} $\tcX$.

An \emph{snc} test configuration for $X$ is a smooth test configuration $\cX$
whose central fiber $\cX_0$ has simple normal crossing support
(but is not necessarily reduced).

When $\cX$ is a test configuration, we define 
the \emph{logarithmic canonical bundle} as
\begin{equation*}
  K^\lo_\cX:=K_\cX+\cX_{0,\red}. 
\end{equation*}
Setting $K^\lo_{\C}:=K_\C+[0]$, we define the 
\emph{relative logarithmic canonical bundle} as
\begin{equation*}
  K^\lo_{\cX/\C}:=K^\lo_\cX-\pi^*K^\lo_{\C}=K_{\cX/\C}+\cX_{0,\red}-\cX_0;
\end{equation*}
this is well behaved under base change $\tau\mapsto\tau^d$,
see~\cite[\S4.4]{BHJ1}.
Despite the terminology, $K_\cX$, $K_{\cX/\C}$, $K^\lo_\cX$ and $K^\lo_{\cX/\C}$ are
only Weil divisor classes in general; they are line bundles when $\cX$ is smooth.
\begin{defi}
  A test configuration $(\cX,\cL)$ for $(X,L)$ consists of a test
  configuration $\cX$ for $X$, together with the following additional
  data:
  \begin{itemize}
  \item[(iv)] a $\C^*$-linearized $\Q$-line bundle $\cL$ on $\cX$;
  \item[(v)] an isomorphism $(\cX_1,\cL_1)\simeq (X,L)$.
  \end{itemize}
\end{defi}
A \emph{pull-back} of a test configuration $(\cX,\cL)$ is a test configuration $(\cX',\cL')$ where $\cX'$ dominates $\cX$ and $\cL'$ is the pull-back of $\cL$. In particular, the \emph{normalization} $(\tcX,\tcL)$ is the pull-back of $(\cX,\cL)$ with $\nu\colon\tcX\to\cX$ the normalization morphism. 

A test configuration $(\cX,\cL)$ is \emph{trivial} if $\cX=X\times\C$
with $\C^*$ acting trivially on $X$. This implies that $(\cX,\cL+c\cX_0)=(X,L)\times\C$ for some constant $c\in\Q$. A test configuration for $(X,L)$ is \emph{almost trivial} if its
normalization is trivial.
 
We say that  $(\cX,\cL)$ is ample (resp.\ semiample, resp.\ nef) when $\cL$ is 
relatively ample (resp.\ relatively semiample, resp.\ nef). The pullback of a
semiample (resp.\ nef) test configuration is semiample (resp.\ nef). 

If $L$ is ample, then for every 
semiample test configuration $(\cX,\cL)$ there exists a unique
ample test configuration $(\cX_\amp,\cL_\amp)$ that is dominated by
$(\cX,\cL)$ and satisfies $\mu_*\cO_\cX=\cO_{\cX_\amp}$, where 
$\mu\colon\cX\to\cX_\amp$ is the canonical morphism; 
see~\cite[Proposition~2.17]{BHJ1}.

Note that, while $\cX$ can often be chosen smooth, $\cX_\amp$ will
not be smooth, in general. It is, however, normal whenever $\cX$ is. 
%
%
\subsection{One-parameter subgroups}\label{S109}
Suppose $L$ is ample. 
Ample test configurations are then essentially equivalent to one-parameter
degenerations of $X$. See~\cite[\S2.3]{BHJ1} for details on what follows.

Fix $m\ge 1$ such that $mL$ is very ample, and consider the
corresponding closed embedding $X\hookrightarrow\P^{N_m-1}$ with
$N_m:=h^0(X,mL)$.  Then every $1$-parameter subgroup ($1$-PS for
short) $\la\colon\C^*\to\GL(N_m,\C)$ induces an ample test configuration
$(\cX_\la,\cL_\la)$ for $(X,L)$. By definition, $\cX_\la$ is the
Zariski closure in $\P V\times\C$ of the image of the closed embedding
$X\times\C^*\hookrightarrow\P V\times\C^*$ mapping $(x,\tau)$ to
$(\la(\tau)x,\tau)$. Note that $(\cX_\la,\cL_\la)$ is trivial iff $\la$ is a multiple of the identity.
We emphasize that $\cX_\la$ is not normal in general.

In fact, every ample test configuration may be obtained as above.
Using one-parameter subgroups, we can produce test configurations that are
almost trivial but not trivial, as observed in~\cite[Remark~5]{LX}. 
See~\cite[Proposition~2.12]{BHJ1} for an elementary proof of the following result.
\begin{prop}\label{prop:almost} 
  For every $m$ divisible enough, there exists a
  $1$-PS $\la\colon\C^*\to\GL(N_m, \C)$ such that the 
  test configuration $(\cX_\la,\cL_\la)$ is nontrivial but almost trivial.
\end{prop}
%
%
\subsection{Valuations and log discrepancies}\label{S201}
By a \emph{valuation on $X$} we mean a real-valued valuation $v$ on the function field $\C(X)$ (trivial on the ground field $\C$). The \emph{trivial valuation} $v_\triv$ is 
defined by $v_\triv(f)=0$ for $f\in\C(X)^*$. A valuation $v$ is
\emph{divisorial} if it is of the form $v=c\ord_F$, where
$c\in\Q_{>0}$ and $F$ is a prime divisor on a projective normal
variety $Y$ admitting a birational morphism onto $X$.
We denote by $\Xdiv$ the set of valuations on $X$ that are either 
divisorial or trivial, and equip it with the weakest topology such that
$v\mapsto v(f)$ is continuous for every $f\in\C(X)^*$. 

The \emph{log discrepancy} $A_X(v)$ of a valuation in $\Xdiv$ is
defined as follows. First, $A_X(v_\triv)=0$. 
For $v=c\ord_F$ a divisorial valuation as above, we set $A_X=c(1+\ord_F(K_{Y/X}))$, 
where $K_{Y/X}$ is the relative canonical (Weil) divisor. 

Now consider a normal test configuration $\cX$ of $X$. Since 
$\C(\cX)\simeq\C(X)(\tau)$, any valuation $w$ on $\cX$
restricts to a valuation $r(w)$ on $X$.
Let $E$ be an
irreducible component of the central fiber $\cX_0=\sum b_EE$.
Then $\ord_E$ is a $\C^*$-invariant divisorial valuation on $\C(\cX)$
and satisfies $\ord_E(t)=b_E$. If we set 
$v_E:=r(b_E^{-1}\ord_E)$,
then $v_E$ is a valuation in $\Xdiv$. Conversely, every valuation
$v\in\Xdiv$ has a unique $\C^*$-invariant preimage $w$ under $r$
normalized by $w(\tau)=1$, and $w$ is associated to an irreducible
component of the central fiber of some test configuration for $X$, cf.~\cite[Theorem 4.6]{BHJ1}. 

Note that $\ord_E$ is a divisorial valuation on $X\times\C$. 
By~\cite[Proposition 4.11]{BHJ1}, the log discrepancies of $\ord_E$ and $v_E$ are related as follows:
$A_{X\times\C}(\ord_E)=b_E(1+A_X(v_E))$.
%
%
\subsection{Compactifications}
For some purposes it is convenient to compactify test configurations. The following notion
provides a canonical way of doing so.
\begin{defi}\label{defi:comp} 
  The \emph{compactification} $\bar\cX$ of a test configuration $\cX$ for $X$ is defined by 
  gluing together $\cX$ and $X\times(\P^1\setminus\{0\})$ along their respective open subsets 
  $\cX\setminus\cX_0$ and $X\times(\C\setminus\{0\})$, using the canonical $\C^*$-equivariant 
  isomorphism $\cX\setminus\cX_0\simeq X\times(\C\setminus\{0\})$. 
\end{defi}
The compactification $\bar\cX$ comes with a $\C^*$-equivariant flat morphism $\bar\cX\to\P^1$, 
still denoted by $\pi$. By construction, $\pi^{-1}(\P^1\setminus\{0\})$ is $\C^*$-equivariantly 
isomorphic to $X\times(\P^1\setminus\{0\})$ over $\P^1\setminus\{0\}$. 

Similarly, a test configuration $(\cX,\cL)$ for $(X,L)$ admits a compactification $(\bar\cX,\bar\cL)$, where $\bar\cL$ is a $\C^*$-linearized $\Q$-line bundle on $\bar\cX$. Note that $\bar\cL$ is relatively (semi)ample iff $\cL$ is. 

The relative canonical differential and relative canonical differential are now defined by 
\begin{equation*}
  K_{\bar\cX/\P^1}:=K_{\bar\cX}-\pi^*K_{\P^1}
\end{equation*}
\begin{equation*}
  K^\lo_{\bar\cX/\P^1}:=K^\lo_{\bar\cX}-\pi^*K^\lo_{\P^1}=K_{\bar\cX/\P^1}+\cX_{0,\red}-\cX_0.
\end{equation*}
%
%
\subsection{Non-Archimedean metrics}\label{S102}
Following~\cite[\S 6]{BHJ1} (see also~\cite{trivval}) we introduce:
\begin{defi}\label{defi:equiv} 
  Two test configurations $(\cX_1,\cL_1)$, $(\cX_2,\cL_2)$ for $(X,L)$ are \emph{equivalent} if there exists a test configuration $(\cX_3,\cL_3)$ that is a pull-back of both $(\cX_1,\cL_1)$ and $(\cX_2,\cL_2)$. 
An equivalence class is called a \emph{non-Archimedean metric} on $L$, and is denoted by $\phi$. We denote by $\phi_\triv$ the equivalence class of the trivial test configuration $(X,L)\times\C$.\end{defi} 
A non-Archimedean metric $\phi$ is called \emph{semipositive} if some
(or, equivalently, any) representative $(\cX,\cL)$ of $\phi$ is
nef. Note that this implies that $L$ is nef.

When $L$ is ample, we say that a non-Archimedean metric $\phi$ on $L$
is \emph{positive} if some (or, equivalently, any) representative 
$(\cX,\cL)$ of $\phi$ is semiample. 
We denote by $\cH^{\NA}$ the set of all non-Archimedean positive
metrics on $L$. 
By~\cite[Lemma~6.3]{BHJ1}, every $\phi\in\cH^{\NA}$ is represented 
by a unique normal, ample test configuration. 

\smallskip
The set of non-Archimedean metrics on a line bundle $L$ 
admits two natural operations: 
\begin{itemize}
\item[(i)] a \emph{translation action} of $\Q$, denoted by $\phi\mapsto\phi+c$, and induced by $(\cX,\cL)\mapsto (\cX,\cL+c\cX_0)$;
\item[(ii)] a \emph{scaling action} of the semigroup $\N^*$ of positive integers, denoted by $\phi\mapsto \phi_d$ and induced by the base change of $(\cX,\cL)$ by $\tau\mapsto\tau^d$. 
\end{itemize}
When $L$ is ample (resp.\ nef) these operations preserve the set of 
positive (resp.\ semipositive) metrics. 
The trivial metric $\phi_\triv$ is fixed by the scaling action.

As in~\S\ref{S107} we use additive notation for non-Archimedean
metrics. A non-Archimedean metric on $\cO_X$ induces a bounded (and 
continuous) function on $\Xdiv$.
\begin{rmk}
  As explained in~\cite[\S6.8]{BHJ1}, a non-Archimedean metric $\phi$ on $L$, 
  as defined above, can be
  viewed as a metric on the Berkovich analytification~\cite{BerkBook} of $L$ with
  respect to the trivial absolute value on the ground field $\C$. 
  See also~\cite{trivval} for a more systematic analysis, itself
  building upon~\cite{siminag,nama}.
\end{rmk}
%
%
\subsection{Intersection numbers and Monge-Amp\`ere measures}
Following~\cite[\S6.6]{BHJ1} we define the intersection number $(\phi_0\cdot\ldots\cdot\phi_n)$
of non-Archimedean metrics $\phi_0,\dots,\phi_n$
on line bundles $L_0,\dots,L_n$ on $X$ as follows. 
Pick representatives $(\cX,\cL_i)$ of $\phi_i$, $0\le i\le n$, with
the same test configuration $\cX$ for $X$ and set 
\begin{equation*}
  (\phi_0\cdot\ldots\cdot\phi_n):=(\bar\cL_0\cdot\ldots\cdot\bar\cL_n),
\end{equation*}
where $(\bar\cX,\bar\cL_i)$ is the compactification of $(\cX,\cL_i)$.
It follows from the projection formula that this does not depend of the choice of the $\cL_i$.
Note that $(\phi_\triv^{n+1})=0$.
When $L_0=\cO_X$, so that $\cL_0=\cO_X(D)$ for a $\Q$-Cartier
$\Q$-divisor $D=\sum r_EE$ supported on $\cX_0$, we can compute the intersection
number as 
$(\phi_0\cdot\ldots\cdot\phi_n)=\sum_Er_E(\cL_1|_E\cdot\ldots\cdot\cL_n|_E)$.

To a non-Archimedean metric $\phi$ on a big and nef line bundle $L$ on $X$ 
we associate, as in~\cite[\S6.7]{BHJ1}, 
a signed finite atomic \emph{Monge-Amp\`ere measure} 
on $\Xdiv$. Pick a representative $(\cX,\cL_i)$ of $\phi$, and set 
\begin{equation*}
  \MA^\NA(\phi)=V^{-1}\sum_Eb_E(\cL|_E^n)\delta_{v_E},
\end{equation*}
where $E$ ranges over irreducible components of $\cX_0=\sum_Eb_EE$,
$v_E=r(b_E^{-1}\ord_E)\in\Xdiv$, and $V=(L^n)$.
When the $\phi_i$ are semipositive, the mixed Monge-Amp\`ere measure is
a probability measure.
%
%
\subsection{Functionals on non-Archimedean metrics}\label{S111}
Following~\cite[\S 7]{BHJ1} we define non-Archi\-me\-dean analogues of the
functionals considered in~\S\ref{S101}.
Fix a line bundle $L$.
\begin{defi}\label{D201}
  Let $W$ be a set of non-Archimedean metrics on $L$ that is closed
  under translation and scaling.
  A functional $F\colon W\to\R$ is
  \begin{itemize}
  \item[(i)]
    \emph{homogeneous} if $F(\phi_d)=d F(\phi)$ for $\phi\in W$and $d\in\N^*$;
  \item[(ii)]
    \emph{translation invariant} if $F(\phi+c)=F(\phi)$ for $\phi\in W$ and $c\in\Q$. 
  \end{itemize}
\end{defi}
When $L$ is ample, a functional $F$ on $\cH^{\NA}$ may be viewed
as a function $F(\cX,\cL)$ on the set of all semiample test
configurations $(\cX,\cL)$ that is invariant under pull-back,
\ie $F(\cX',\cL')=F(\cX,\cL)$ whenever $(\cX',\cL')$ is a pull-back of
a $(\cX,\cL)$ (and, in particular, invariant under
normalization). Homogeneity amounts to $F(\cX_d,\cL_d)=d\,F(\cX,\cL)$ for all $d\in\N^*$, and translation invariance to $F(\cX,\cL)=F(\cX,\cL+c\cX_0)$ for all $c\in\Q$. 

\smallskip
For each non-Archimedean metric $\phi$ on $L$, choose a normal representative
$(\cX,\cL)$ that dominates $X\times\C$ via $\rho\colon\cX\to X\times\C$. Then
$\cL=\rho^*(L\times\C)+D$ for a uniquely determined $\Q$-Cartier divisor $D$
supported on $\cX_0$. Write $\cX_0=\sum_E b_E E$ 
and let $(\bar\cX,\bar\cL)$ be the compactification of $(\cX,\cL)$.

In this notation, we may describe our list of non-Archimedean
functionals as follows. 
Assume $L$ is big and nef.
Let $\phi_\triv$ and $\psi_\triv$ be the trivial metrics on $L$ and
$K_X$, respectively.
\begin{itemize}
\item[(i)] 
  The \emph{non-Archimedean Monge-Amp\`ere energy} of $\phi$ is
  \begin{align*}
    E^{\NA}(\phi)
    :&=\frac{(\phi^{n+1})}{(n+1)V}\\
    &=\frac{\left(\bar\cL^{n+1}\right)}{(n+1)V}. 
  \end{align*}
\item[(ii)] The \emph{non-Archimedean} \emph{$I$-functional} and 
  \emph{$J$-functional} are given by 
  \begin{align*}
    I^{\NA}(\phi) 
    :&=V^{-1}(\phi\cdot\phi_\triv^n)-V^{-1}((\phi-\phi_\triv)\cdot\phi^n)\\
     &=V^{-1}(\bar\cL\cdot(\rho^*(L\times\P^1)^n)-V^{-1}(D\cdot\bar\cL^n).
  \end{align*}
  and
  \begin{align*}
    J^{\NA}(\phi) 
    :&=V^{-1}(\phi\cdot\phi_\triv^n)-E^{\NA}(\phi)\\
     &=\frac1V(\bar\cL\cdot(\rho^*(L\times\P^1)^n)-\frac1{(n+1)V}(\bar\cL^{n+1}).
  \end{align*}
\item[(iii)] The \emph{non-Archimedean Ricci energy} is
\begin{align*}
  R^{\NA}(\phi)
  :&=V^{-1}(\psi_\triv\cdot\phi^n)\\
   &=V^{-1}\left(\rho^*K^\lo_{X\times\P^1/\P^1}\cdot\bar\cL^n\right).
\end{align*}
\item[(iv)] The \emph{non-Archimedean entropy} is
  \begin{align*}
    H^{\NA}(\phi)
    :&=\int_{\Xdiv}A_X(v)\MA^{\NA}(\phi)\\
    &=V^{-1}\left(K^\lo_{\bar\cX/\P^1}\cdot\bar\cL^n\right)
      -V^{-1}\left(\rho^*K^\lo_{X\times\P^1/\P^1}\cdot\bar\cL^n\right).
 \end{align*}
\item[(v)] 
  The \emph{non-Archimedean Mabuchi functional} (or K-energy) is
  \begin{align*}
    M^{\NA}(\phi)
    :&=H^{\NA}(\phi)+R^{\NA}(\phi)+\bar S E^{\NA}(\phi)\\
     &=V^{-1}\left(K^\lo_{\bar\cX/\P^1}\cdot\bar\cL^n\right)
       +\frac{\bar S}{(n+1)V}\left(\bar\cL^{n+1}\right). 
  \end{align*}
\end{itemize}
Note the resemblance to the formulas in~\S\ref{S108}.
All of these functionals are homogeneous. 
They are also translation invariant, except for $E^{\NA}$ and $R^\NA$, which satisfy
\begin{equation}\label{equ:transE}
  E^{\NA}(\phi+c)=E^{\NA}(\phi)+c 
  \quad\text{and}\quad
  R^{\NA}(\phi+c)=R^{\NA}(\phi)-\bar S c 
\end{equation}
for all $\phi\in\cH^{\NA}$ and $c\in\Q$. 

The functionals $I^{\NA}$, $J^\NA$ and $I^\NA-J^\NA$ are comparable
on semipositive metrics in the same way as~\eqref{e106}.
By~\cite[Lemma 7.7, Theorem~5.16]{BHJ1}, when $\phi$ is positive, the first term in the definition of $J^{\NA}$  satisfies
\begin{equation*}
  V^{-1}(\phi\cdot\phi_\triv^n) 
  =(\phi-\phi_\triv)(v_\triv)
  =\max_{\Xdiv}(\phi-\phi_\triv) 
  =\max_Eb_E^{-1}\ord_E(D). 
\end{equation*}

Further, $J^{\NA}(\phi)\ge 0$, with equality iff
$\phi=\phi_\triv+c$ for some $c\in\Q$, and $J^\NA$ is comparable to
both a natural $L^1$-norm and the minimum norm in the sense of 
Dervan~\cite{Der1}, see~\cite[Theorem~7.9, Remark~7.12]{BHJ1}.
For a normal ample test configuration $(\cX, \cL)$ representing $\phi \in \cH^{\NA}$ we also
denote the J-norm by $J^{\NA}(\cX, \cL)$. 
%
%
\subsection{The Donaldson-Futaki invariant}
As explained in~\cite{BHJ1}, the non-Archimedean Mabuchi
functional is closely related to the Donaldson-Futaki invariant. 
We have
\begin{prop} 
  Assume $L$ is ample. Let $\phi\in\cH^\NA$ be the class of an ample
  test configuration $(\cX,\cL)$ for $(X,L)$, and denote by $(\tcX,\tcL)$ its
  normalization,  which is thus the unique normal, ample representative of $\phi$. Then 
  \begin{align}
    M^{\NA}(\phi)&=\DF(\tcX,\tcL)-V^{-1}\left((\tcX_0-\tcX_{0,\red})\cdot\tcL^n\right)\label{e104}\\
    \DF(\cX,\cL)&=\DF(\tcX,\tcL)+2 V^{-1}\sum_E m_E\left(E\cdot\cL^n\right),\label{e105}
  \end{align}
  where $E$ ranges over the irreducible components of $\cX_0$ 
  contained in the singular locus of $\cX$ and $m_E\in\N^*$ 
  is the length of $\left(\nu_*\cO_{\tcX}\right)/\cO_{\cX}$ 
  at the generic point of $E$, with $\nu\colon\tcX\to\cX$ the normalization. 
  
  In particular, $\DF(\cX,\cL)\ge M^{\NA}(\phi)$, and equality holds 
  iff $\cX$ is regular in codimension one and $\cX_0$ is generically reduced. 
\end{prop}
Indeed,~\eqref{e104} and~\eqref{e105} follow from the discussion in~\cite[\S7.3]{BHJ1} and 
from~\cite[Proposition~3.15]{BHJ1}, respectively.
Note that intersection theoretic formulas for the Donaldson-Futaki invariant appeared 
already in~\cite{Wan12} and~\cite{Oda13}.

For a general non-Archimedean metric $\phi$ on $L$ we can define
\begin{align*}
  \DF(\phi)
  &=M^{\NA}(\phi)+V^{-1}\left((\cX_0-\cX_{0,\red})\cdot\bar\cL^n\right)\\
  &=V^{-1}\left(K_{\bar\cX/\P^1}\cdot\bar\cL^n\right) 
    +\frac{\bar S}{(n+1)V}\left(\bar\cL^{n+1}\right)
\end{align*}
for any normal representative $(\cX,\cL)$ of $\phi$.
Clearly $M^\NA(\phi)\le\DF(\phi)$ when $\phi$ is semipositive.
%
%
\subsection{The non-Archimedean Ding functional~\cite[\S7.7]{BHJ1}}
Suppose $X$ is weakly Fano, that is, $L:=-K_X$ is big and nef. In this case, we
define the \emph{non-Archimedean Ding functional} on the space of
non-Archimedean metrics on $L$ by 
\begin{equation*}
  D^\NA(\phi)=L^\NA(\phi)-E^\NA(\phi),
\end{equation*}
where $L^{\NA}$ is defined by
\begin{equation*}
  L^\NA(\phi)
  =\inf_v\left(A_X(v)+(\phi-\phi_{\triv})(v)\right),
\end{equation*}
the infimum taken over all valuations $v$ on $X$ that are divisorial
or trivial.
Recall from~\S\ref{S102} that $\phi-\phi_\triv$ is a non-Archimedean 
metric on $\cO_X$ and induces a bounded function on divisorial
valuations.
Note that $L^\NA(\phi+c)=L^\NA(\phi)+c$;
hence $D^\NA$ is translation invariant.

We always have $D^\NA\le J^\NA$, see~\cite[Proposition~7.28]{BHJ1}. When 
$\phi$ is semipositive, we have $D^\NA(\phi)\le M^\NA(\phi)$, 
see~\cite[Proposition~7.32]{BHJ1}.
%
%
\subsection{Uniform K-stability}
As in~\cite[\S8]{BHJ1} we make the following definition.
\begin{defi}\label{D202}
  A polarized complex manifold $(X,L)$ is \emph{uniformly K-stable} 
  if there exists a constant $\delta>0$ such that the following 
  equivalent conditions hold. 
  \begin{itemize}
  \item[(i)]  
    $M^{\NA}(\phi)\ge\delta J^{\NA}(\phi)$ for every $\phi\in\cH^{\NA}(L)$;
  \item[(ii)]
    $\DF(\phi)\ge\delta J^{\NA}(\phi)$ for every $\phi\in\cH^{\NA}(L)$; 
  \item[(iii)]
    $\DF(\cX, \cL) \geq \delta J^{\NA}(\cX, \cL)$ for any normal 
    ample test configuration $(\cX, \cL)$. 
    \end{itemize} 
\end{defi} 
Here the equivalence between~(ii) and~(iii) is definitional, 
and~(i)$\implies$(ii) follows immediately from $\DF\le M^\NA$.
The implication~(ii)$\implies$(i) follows from the homogeneity of 
$M^\NA$ together with the fact that $\DF(\phi_d)=M^\NA(\phi_d)$
for $d$ sufficiently divisible. 
See~\cite[Proposition~8.2]{BHJ1} for details.

The fact that $J^\NA(\phi)=0$ iff $\phi=\phi_\triv+c$ implies that 
uniform K-stability is stronger than K-stability as introduced by
\cite{Tian97,Don2}.
Our notion of uniform K-stability is equivalent to uniform K-stability defined either with respect to the $L^1$-norm or the minimum norm in the 
sense of~\cite{Der1}, see~\cite[Remark~8.3]{BHJ1}. 

\smallskip
In the Fano case, uniform K-stability is further equivalent to 
\emph{uniform Ding stability}:
\begin{thm}
  Assume $L:=-K_X$ is ample and fix a number $\delta$ with
  $0\le\delta\le 1$. Then the following conditions are equivalent:
  \begin{itemize}
  \item[(i)]  
    $M^{\NA}\ge\delta J^{\NA}$ on $\cH^\NA$;
  \item[(ii)]
    $D^{\NA}\ge\delta J^{\NA}$ on $\cH^\NA$.
  \end{itemize}
\end{thm}
This is proved in~\cite{BBJ15} using the Minimal Model Program as in~\cite{LX}.
See~\cite{Fuj16} for a more general result, and also~\cite{Fuj15}.
%
%
%
%
\section{Non-Archimedean limits}\label{S105}
In this section we prove Theorem~A and Corollary~B.
%
%
\subsection{Rays of metrics and non-Archimedean limits}
For any line bundle $L$ on $X$, 
there is a bijection between smooth rays $(\phi^s)_{s>0}$
of metrics on $L$ and $S^1$-invariant smooth metrics $\Phi$ on
the pull-back of $L$ to $X\times\D^*$, 
with $\D^*=\D^*_1\subset\C$ the punctured unit disc. 
The restriction of $\Phi$ to $\cX_\tau$ for
$\tau\in\Delta^*$ is given by pullback of $\phi^{\log|\tau|^{-1}}$
under the map $\cX_\tau\to X$ given by the $\C^*$-action.
Similarly, smooth rays $(\phi^s)_{s>s_0}$ correspond to 
$S^1$-invariant smooth metrics on the pull-back of $L$ to $X\times\Delta_{r_0}^*$,
with $r_0=e^{-s_0}$.

A \emph{subgeodesic ray} is a ray $(\phi^s)$ whose corresponding
metric $\Phi$ is semipositive. Such rays can of course only exist when 
$L$ is nef.
\begin{defi} 
  We say that a smooth ray $(\phi^s)$ admits a non-Archimedean metric 
  $\phi^\NA$ as \emph{non-Archimedean limit} if there exists 
  a test configuration $(\cX,\cL)$ representing $\phi^\NA$ such that 
  the metric $\Phi$ on $L\times\D^*$ corresponding to $(\phi^s)_s$
  extends to a smooth metric on $\cL$ over $\D$. 
\end{defi}
In other words, a non-Archimedean limit exists iff $\Phi$ has \emph{analytic singularities} along $X\times\{0\}$, \ie splits into a smooth part and a divisorial part after pulling-back to a blow-up. 
\begin{lem}\label{lem:welldef}
  Given a ray $(\phi^s)_s$ in $\cH$, 
  the non-Archimedean limit $\phi^{\NA}\in\cH^{\NA}$ is unique, if it exists. 
\end{lem} 
\begin{proof}
  Let $\psi_1$ and $\psi_2$ be non-Archimedean limits of $(\phi^s)_s$
  and let $\Phi$ be the smooth metric on $L\times\Delta^*$ defined by the
  ray $(\phi^s)$.
  For $i=1,2$, pick a representative $(\cX_i,\cL_i)$ of $\psi_i$ 
  such that $\Phi$ extends as a smooth metric on $\cL_i$ over $\D$.
  After replacing $(\cX_i,\cL_i)$ by suitable pullbacks, we 
  may assume $\cX_1=\cX_2=\colon\cX$ and that $\cX$ is normal.
  Then $\cL_2=\cL_1+D$ for a $\Q$-divisor $D$ supported
  on $\cX_0$. Now a smooth metric on $\cL_1$ induces a singular metric
  on $\cL_1+D$ that is smooth iff $D=0$. Hence $\cL_1=\cL_2$, so that 
  $\psi_1=\psi_2$.
\end{proof}

\begin{rmk}
  Following~\cite[\S2]{BerkHodge} (see also~\cite{amoebae,konsoib}) 
  one can construct a
  compact Hausdorff space $X^\An$ fibering over the interval $[0,1]$
  such that the fiber $X^\An_\rho$ over any point $\rho\in(0,1]$ 
  is homeomorphic to the complex manifold $X$, 
  and the fiber $X^\An_0$ over $0$ is homeomorphic to 
  the Berkovich analytification of $X$ with respect to the trivial
  norm on $\C$. Similarly, the line bundle $L$ induces a line bundle
  $L^\An$ over $X^{\An}$.
  If a ray $(\phi^s)_{s>0}$ admits a non-Archimedean limit
  $\phi^\NA$, then it induces a continuous metric on $L^{\An}$ whose
  restriction to $L^\An_\rho$ is given by $\phi^{\log\rho^{-1}}$
  and whose restriction to $X^\an_0$ is given by $\phi^\NA$.
  In this way, $\phi^\NA$ is indeed the limit of $\phi^s$ as $s\to\infty$.
\end{rmk}
%
%
\subsection{Non-Archimedean limits of functionals}
For the rest of~\S\ref{S105}, assume that $L$ is ample.
\begin{defi}\label{defi:NAlimit} 
  A functional $F\colon\cH\to\R$ admits a functional
  $F^{\NA}\colon\cH^{\NA}\to\R$ as a \emph{non-Archimedean limit} 
  if, for every smooth subgeodesic ray $(\phi^s)$ in $\cH$ admitting a
  non-Archimedean limit $\phi^\NA\in\cH^\NA$, we have 
  \begin{equation}\label{equ:Mslope}
    \lim_{s\to+\infty}\frac{F(\phi^s)}{s}=F^{\NA}(\phi^\NA).
  \end{equation}
\end{defi}
\begin{prop}\label{prop:NAlim} 
  If $F\colon\cH\to\R$ admits a non-Archimedean limit
  $F^{\NA}\colon\cH^{\NA}\to\R$, then $F^{\NA}$ is homogeneous.
\end{prop}
\begin{proof} 
  Consider a semiample test configuration $(\cX,\cL)$ representing a
  non-Archimedean metric $\phi^\NA\in\cH^\NA$, and let 
  $(\phi^s)_s$ be a smooth subgeodesic ray admitting $\phi^\NA$ as a non-Archimedean limit.
  For $d\ge1$, let $(\cX_d,\cL_d)$ be the normalized base change 
  induced by $\tau\to\tau^d$. The associated non-Archimedean metric 
  $\phi_d^\NA$ is then the non-Archimedean limit of the subgeodesic ray
  $(\phi^{ds})$, so 
  $\lim_{s\to\infty}s^{-1}F(\phi_{ds})=F^\NA(\phi^\NA_d)$.
  On the other hand, we clearly have 
  $\lim_{s\to\infty}(ds)^{-1}F(\phi^{ds})=\lim_{s\to\infty}s^{-1}F(\phi^s)=F^\NA(\phi^\NA)$.
  The result follows.
\end{proof}
%
%
\subsection{Asymptotics of the functionals}\label{S112}
The following result immediately implies Theorem~A and Corollary~B.
\begin{thm}\label{thm:asymfunc} 
  The functionals $E$, $H$, $I$, $J$, $M$ and $R$ 
  on $\cH$ admit non-Archimedean limits
  on $\cH^{\NA}$ given, respectively, by 
  $E^\NA$, $H^\NA$, $I^\NA$, $J^\NA$, $M^\NA$ and $R^\NA$.
\end{thm}
In addition, we have the following result due to Berman~\cite[Proposition~3.8]{Berm16}.
See also~\cite[Theorem 3.1]{BBJ15} for a more general result.
\begin{thm}\label{asymDing}
  If $L:=-K_X$ is ample, then the Ding functional $D$ on $\cH$ admits 
  $D^\NA$ on $\cH^\NA$ as non-Archimedean limit.
\end{thm}
\begin{rmk}\label{R202}
  In~\S\ref{S110} we will extend the two previous results to the logarithmic setting
  and with relaxed positivity assumptions.
\end{rmk}

The main tool in the proof of Theorem~\ref{thm:asymfunc} is the following result (compare \cite[Lemma 6]{PRS}). 
\begin{lem}\label{lem:deligne} 
  For $i=0,\dots,n$, let $L_i$ be a line bundle on $X$ with a
  smooth reference metric $\phi_{i,\re}$. 
  Let also $(\cX,\cL_i)$ be a smooth test configuration for $(X,L_i)$, 
  $\Phi_i$ an $S^1$-invariant smooth metric on $\cL_i$ near $\cX_0$,
  and denote by $(\phi_i^s)$ the corresponding ray of smooth
  metrics on $L_i$. Then 
  \begin{equation*}
    \langle\phi_0^s,\dots,\phi_n^s\rangle_X
    -\langle\phi_{0,\re},\dots,\phi_{n,\re}\rangle_X
    =s\left(\bar\cL_0\cdot\ldots\cdot\bar\cL_n\right)+O(1)
  \end{equation*}
  as $s\to\infty$. Here $(\bar\cX,\bar\cL_i)$ is the compactification of
  $(\cX,\cL_i)$ for $0\le i\le n$ and
  $\langle\cdot,\ldots,\cdot\rangle_X$ denotes the Deligne pairing
  with respect to the constant morphism $X\to\{\mathrm{pt}\}$.
\end{lem}
\begin{proof} 
  The Deligne pairing $F:=\langle\cL_0,\dots,\cL_n\rangle_{\cX/\C}$ is a
  line bundle on $\C$, endowed with a $\C^*$-action and a canonical identification of 
  its fiber at $\tau=1$ with the complex line $\langle L_0,\dots,L_n\rangle_X$. It extends to a 
  line bundle $\langle\bar\cL_0,\dots,\bar\cL_n\rangle_{\bar\cX/\P^1}$ on
  $\P^1$ that is $\C^*$-equivariantly trivial on $\P^1\smallsetminus\{0\}$.
  Denoting by $w\in\Z$ the weight of the $\C^*$-action on the fiber at $0$, we have 
  \begin{equation*}
    w
    =\deg\langle\bar\cL_0,\dots,\bar\cL_n\rangle_{\bar\cX/\P^1}
    =\left(\bar\cL_0,\dots,\bar\cL_n\right). 
  \end{equation*}
  Pick a nonzero vector $v\in F_1=\langle L_0,\dots,L_n\rangle_X$.
  The $\C^*$-action produces a section $\tau\mapsto\tau\cdot v$ of $F$ on 
  $\C^*$, and
  $\sigma:=\tau^{-w}(\tau\cdot v)$ is a nowhere vanishing section of $F$
  on $\C$, see~\cite[Corollary~1.4]{BHJ1}. 

  Since the metrics $\Phi_i$ are smooth and $S^1$-invariant,
  $\Psi:=\langle\Phi_0,\dots,\Phi_n\rangle_{\cX/\C}$ 
  is a continuous $S^1$-invariant metric on $F$ near $0\in\C$.
  Hence the function $\log|\sigma|_\Psi$ is bounded near
  $0\in\C$.

  The $S^1$-invariant metric $\Psi$ defines a ray $(\psi^s)$ of
  metrics on the line $F_1$ through
  $|v|_{\psi^s}=|\tau\cdot v|_{\Psi_\tau}$, for $s=\log|\tau|^{-1}$,
  where $\Psi_\tau$ is the restriction of $\Psi$ to $F_\tau$.
  Thus
  \begin{equation*}
    \log|v|_{\psi^s}
    =\log|\tau\cdot v|_{\Psi_\tau}
    =w\log|\tau|+\log|\sigma|_{\Psi_\tau}
    =-sw+O(1).
  \end{equation*}
  By functoriality, the metric $\psi^s$ on $F_1$ is nothing but the
  Deligne pairing $\langle\phi^s_0,\dots,\phi^s_n\rangle$.
  If we set $\psi_\re=\langle\phi_{0,\re},\dots,\phi_{n,\re}\rangle_X$,
  it therefore follows that 
  \begin{equation*}
    \langle\phi_0^s,\dots,\phi_n^s\rangle_X
    -\langle\phi_{0,\re},\dots,\phi_{n,\re}\rangle_X
    =\log|v|_{\psi_\re}-\log|v|_{\psi^s}
    =sw+O(1),
  \end{equation*}
  which completes the proof.  
\end{proof}
\begin{proof}[Proof of Theorem~\ref{thm:asymfunc}] 
  Let $(\phi^s)_s$ be a smooth subgeodesic ray in $\cH$ admitting a 
  non-Archi\-medean limit $\phi^\NA\in\cH^\NA$. Pick a
  test configuration $(\cX,\cL)$ representing $\phi^\NA$ such that
  $\cX$ is smooth and $\cX_0$ has snc support. Thus $\cL$ is relatively semiample and $(\phi^s)_s$ corresponds to a 
  smooth $S^1$-invariant semipositive metric $\Phi$ on $\cL$ over $\Delta$.
  By Lemma~\ref{lem:funcdel}, we have 
  \begin{equation*}
    (n+1)V\left(E(\phi^s)-E(\phi_{\re})\right)
    =\langle\phi^s,\dots,\phi^s\rangle_X-\langle\phi_{\re},\dots,\phi_{\re}\rangle_X. 
  \end{equation*}
  Using Lemma~\ref{lem:deligne}, it follows that 
  \begin{equation*}
    \lim_{s\to+\infty}\frac{E(\phi^s)}{s}=\frac{\left(\bar\cL^{n+1}\right)}{(n+1)V}=E^{\NA}(\phi^{\NA}), 
  \end{equation*}
  which proves the result for the Monge-Amp\`ere energy $E$. 
  The case of the functionals $I$, $J$ and $R$ is similarly 
  a direct consequence of Lemma~\ref{lem:funcdel} and Lemma~\ref{lem:deligne}. 
  In view of the Chen-Tian formulas for $M$ and $M^{\NA}$, 
  it remains to consider the case of the entropy functional $H$. 
  In fact, it turns out to be easier to treat the functional $H+R$.

  By Lemma~\ref{lem:funcdel} we have
  \begin{align*}
    V(H(\phi^s)+R(\phi^s))
    &=\langle\half\log\MA(\phi^s),\phi^s,\dots,\phi^s\rangle_X 
      -\langle\p_\re,\phi_\re,\dots,\phi_\re\rangle_X,
  \end{align*}
  where $\p_\re=\half\log\MA(\phi_\re)$, 
  so we must show that 
  \begin{equation}\label{equ:smallo}
    \langle\half\log\MA(\phi^s),\phi^s,\dots,\phi^s\rangle_X
    -\langle\p_\re,\phi_\re,\dots,\phi_\re\rangle_X
    =s\left(K^\lo_{\bar\cX/\P^1}\cdot\bar\cL^n\right)+o(s).
  \end{equation}
  The collection of metrics $\half\log\MA(\Phi|_{\cX_\tau})$ with $\tau\ne 0$
  defines a smooth metric $\Psi$ on $K^\lo_{\cX/\C}$ over $\D^*$, 
  but the difficulty here (as opposed to the situation in~\cite{PRS})
  is that $\Psi$ will not a priori extend to a smooth
  (or even locally bounded) metric on $K^\lo_{\cX/\C}$ over $\D$.
  Indeed, since we have assumed that $\cX$ is smooth, there is no reason why 
  $\Phi$ is strictly positive.
  
  Instead, pick a smooth, $S^1$-invariant reference metric $\Psi_\re$ on
  $K^\lo_{\cX/\C}$ over $\D$, and denote by $(\psi^s_\re)_{s>0}$ 
  the corresponding ray of smooth metrics on $K_X$. 
  By Lemma~\ref{lem:deligne} we have 
  \begin{equation*}
    \langle\psi^s_\re,\phi^s,\dots,\phi^s\rangle_X
    -\langle\p_\re,\phi_\re,\dots,\phi_\re\rangle_X
    =s\left(K^\lo_{\bar\cX/\P^1}\cdot\bar\cL^n\right)+O(1). 
  \end{equation*}
  Since 
  \begin{equation*}
    \langle\half\log\MA(\phi^s),\phi^s,\dots,\phi^s\rangle_X
    -\langle\psi^s_\re,\phi^s,\dots,\phi^s\rangle_X
    =\half\int_X\log\left[\frac{\MA(\phi^s)}{e^{2\psi^s_\re}}\right](dd^c\phi^s)^n,
  \end{equation*}
  Theorem~\ref{thm:asymfunc}  is therefore a consequence of the following result.
\end{proof}
\begin{lem}\label{L202}
  We have 
  $\int_X\log\left[\frac{\MA(\phi^s)}{e^{2\psi^s_\re}}\right](dd^c\phi^s)^n=O(\log s)$
  as $s\to\infty$. 
\end{lem}
Let us first prove an estimate of independent interest. 
See~\cite{konsoib} for more precise results.
\begin{lem}\label{lem:estim} 
  Let $\cX$ be an snc test configuration for $X$ and
  $\Psi$ a smooth metric on $K^\lo_{\cX/\C}$ near $\cX_0$.
  Denote by $e^{2\Psi_\tau}$ the induced volume form on $\cX_\tau$ 
  for $\tau\ne 0$. Then 
  \begin{equation}\label{e102}
    \int_{\cX_\tau} e^{2\Psi_\tau}\sim\left(\log|\tau|^{-1}\right)^d
    \quad\text{as $\tau\to 0$},
  \end{equation}
  with $d$ denoting the dimension of the dual complex of $\cX_0$, 
  so that $d+1$ is the largest number of local components of $\cX_0$.   
\end{lem}
Here $A\sim B$ means that $A/B$ is bounded from above and below by
positive constants. 
\begin{proof} 
  Since $\cX_0$ is an snc divisor, every point of $\cX_0$ 
  admits local coordinates $(z_0,\dots,z_n)$ that are defined in a neighborhood of 
  $B:=\left\{|z_i|\le 1\right\}$ and such that  $z_0^{b_0}\dots z_p^{b_p}=\e\tau$ 
  with $0\le p\le n$ and $\e>0$. 
  Here $b_i\in\Z_{>0}$ is the multiplicity of $\cX_0$ along $\{z_i=0\}$.
  The integer $d$ in the statement of the theorem is then the largest 
  such integer $p$. By compactness of $\cX_0$, it will be enough to show that
  \begin{equation*}
    \int_{B\cap\cX_\tau} e^{2\Psi_\tau}\sim\left(\log|\tau|^{-1}\right)^p.
  \end{equation*}
  The holomorphic $n$-form 
  \begin{equation*}
    \eta:=\frac1{p+1}\sum_{j=0}^p\frac{(-1)^j}{b_j}\frac{dz_0}{z_0}
    \wedge\dots\wedge
    \widehat{\frac{dz_j}{z_j}}\wedge\dots\wedge\frac{dz_p}{z_p}\wedge dz_{p+1}
    \wedge\dots\wedge dz_n
  \end{equation*}
  satisfies 
  \begin{equation*}
    \eta\wedge\frac{d\tau}{\tau}
    =\frac{dz_0}{z_0}\wedge\dots\wedge\frac{dz_p}{z_p}
    \wedge dz_{p+1}\wedge\dots\wedge dz_n.
  \end{equation*}
  Thus $\eta$ defines a local frame of $K^\lo_{\cX/\C}$ on $B$, 
  so the holomorphic $n$-form $\eta_\tau:=\eta|_{\cX_\tau}$ satisfies
  \begin{equation*}
    C^{-1}|\eta_\tau|^2
    \le e^{2\Psi_\tau}
    \le C|\eta_\tau|^2
  \end{equation*}
  for a constant $C>0$ independent of $\tau$. Hence it suffices to prove
  $\int_{B\cap\cX_\tau}|\eta_\tau|^2\sim\left(\log|\tau|^{-1}\right)^p$. 
  
  To this end, we parametrize $B\cap\cX_\tau$ in (logarithmic) 
  polar coordinates as follows. Consider the $p$-dimensional simplex 
  \begin{equation*}
    \sigma=\{w\in\R_{\ge0}^{p+1}\mid\sum_{j=0}^pb_jw_j=1\},
  \end{equation*}
  the $p$-dimensional (possibly disconnected) commutative compact Lie group
  \begin{equation*}
    T=\{\theta\in(\R/\Z)^{p+1}\mid\sum_{j=0}^pb_j\theta_j=0\},
  \end{equation*}
  and the polydisc $\bbD^{n-p}\subset\C^{n-p}$. We may cover $\C^*$ 
  by two simply
  connected open sets, on each of which we fix a branch of the complex logarithm. 
  We then define a diffeomorphism $\chi_\tau$ from $\sigma\times T\times\bbD^{n-p}$
  to $B\cap\cX_\tau$ by setting
  \begin{equation*}
    z_j=e^{w_j\log(\e\tau)+2\pi i\theta_j}
    \quad\text{for $0\le j\le p$}.
  \end{equation*}
  A simple computation shows that
  \begin{equation*}
    \chi_\tau^*(|\eta_\tau|^2)
    =\mathrm{const}\left(\log|\e\tau|^{-1}\right)^pdV,
  \end{equation*}
  where $dV$ denotes the natural volume form on $\sigma\times T\times\bbD^{n-p}$.
  It follows that, for $|\tau|\ll 1$,  
  \begin{equation*}
    \int_{B\cap\cX_\tau}
    |\eta_\tau|^2
    \sim\int_{\sigma\times T\times\bbD^{n-p}}\chi_\tau^*(|\eta_\tau|^2)
    \sim\left(\log|\tau|^{-1}\right)^p,
  \end{equation*}
  which completes the proof.
\end{proof}
\begin{proof}[Proof of Lemma~\ref{L202}]
  On the one hand, we have 
  \begin{multline*}
    V^{-1}\int_X\log\left[\frac{\MA(\phi^s)}{e^{2\psi^s_\re}}\right](dd^c\phi^s)^n\\
    =\int_X\log\left[\frac{\MA(\phi^s)}{e^{2\psi^s_\re}/\int_X e^{2\psi^s_\re}}\right]\MA(\phi^s)
    -\log\int_X e^{2\psi^s_\re}\ge-\log\int_X e^{2\psi^s_\re},
  \end{multline*}
  since the first term on the second line 
  is the relative entropy of the probability measure
  $\MA(\phi^s)$ with respect to the probability measure
  $e^{2\psi_\re^s}/\int_X e^{2\psi_\re^s}$.
  By Lemma~\ref{lem:estim} we have 
  $\int_Xe^{2\psi^s_\re}=O(s^d)$, where $0\le d\le n$.
  This gives the lower bound in Lemma~\ref{L202}.

  To get the upper bound, it suffices to prove that the function 
  $g_\tau:=\frac{(dd^c\Phi|_{\cX_\tau})^n}{e^{2\Psi_\tau}}$ on $\cX_\tau$ 
  is uniformly bounded from above.
  Indeed, if $\tau=e^{-s}$, we then see that 
  \begin{equation*}
    \int_X\log\left[\frac{\MA(\phi^s)}{e^{2\psi^s_\re}}\right](dd^c\phi^s)^n
    =\int_{\cX_\tau}(\log V^{-1}+\log g_\tau)(dd^c\Phi|_{\cX_\tau})^n
  \end{equation*}
  is uniformly bounded from above, since $(dd^c\Phi|_{\cX_\tau})^n$ has
  fixed mass $V$ for all $\tau$.

  To bound $g_\tau$ from above, 
  we use local coordinates $(z_j)_0^n$ as in the proof of Lemma~\ref{lem:estim}.
  With the notation in that proof, it suffices to prove that the function
  $(\Omega|_{\cX_\tau})^n/e^{2\Psi_\tau}$ on $\cX_\tau$ 
  is uniformly bounded from above, where 
  $\Omega:=\frac{i}{2}\sum_{j=0}^ndz_j\wedge d\bar{z}_j$.
  Indeed, we have $dd^c\Phi\le C\Omega$ for some constant $C>0$.
  It then further suffices to prove the bound
  \begin{equation}\label{e201}
    i^ndz_0\wedge d\bar{z}_0
    \wedge\dots\wedge
    \widehat{dz_j\wedge d\bar{z}_j}
    \wedge\dots\wedge
    dz_n\wedge dz_n\bigg|_{\cX_\tau}
    \le Ce^{2\Psi_\tau}
  \end{equation}    
  for $0\le j\le p$ and a uniform constant $C>0$.

  To prove~\eqref{e201} we use the logarithmic polar coordinates in the proof 
  of Lemma~\ref{L202}. Namely, if 
  $\chi_\tau\colon\sigma\times T\times\bbD^{n-p}\to B\cap X_\tau$ 
  is the diffeomorphism in that proof, we have 
  \begin{equation*}
    \chi_\tau^*(e^{2\Psi_\tau})\sim(\log|\tau|^{-1})^pdV.
  \end{equation*}
  \begin{equation*}
    \chi_\tau^*( i^ndz_0\wedge d\bar{z}_0
    \wedge\dots\wedge
    \widehat{dz_j\wedge d\bar{z}_j}
    \wedge\dots\wedge
    dz_n\wedge dz_n)
    \sim
    (\log|\tau|^{-1})^p\prod_{0\le l\le p, l\ne j}|z_l|^2dV.
    \end{equation*}
    Thus~\eqref{e201} holds, which completes the proof.
  \end{proof}
%
%
%
%
\section{The logarithmic setting}\label{S110}
In this section we extend, for completeness, 
Theorem~\ref{thm:asymfunc}---and hence Theorem~A
and Corollary~B---to the logarithmic setting.
We will also relax the positivity assumptions used.
Our conventions and notation largely follow~\cite{BBEGZ}.
%
%
\subsection{Preliminaries}\label{S202}
If $X$ is a normal projective variety of dimension $n$, and
$\phi_1,\dots,\phi_n$ are smooth metrics on $\Q$-line bundles 
$L_1,\dots,L_n$ on $X$, then we define 
$dd^c\phi_1\wedge\dots\wedge dd^c\phi_n$ as the pushforward of the measure
$dd^c\phi_1|_{X_{\reg}}\wedge\dots\wedge dd^c\phi_n|_{X_{\reg}}$ from $X_{\reg}$ to $X$.
This is a signed Radon measure of total mass $(L_1\cdot\ldots\cdot L_n)$,
positive if the $\phi_i$ are semipositive.

\smallskip
A \emph{boundary} on $X$ is a Weil $\Q$-divisor $B$ on $X$ such that the Weil $\Q$-divisor 
class
\begin{equation*}
  K_{(X,B)}:=K_X+B
\end{equation*}
is $\Q$-Cartier. Note that $B$ is not necessarily effective. We call $(X,B)$ a \emph{pair}.

The \emph{log discrepancy} of a divisorial valuation $v=c\ord_F$ with respect to 
$(X,B)$ is defined as in~\S\ref{S201}, using
$A_{(X,B)}(v)=c(1+\ord_F(K_{Y/(X,B)}))$. The pair $(X,B)$ is \emph{subklt} if 
$A_{(X,B)}(v)>0$ for all (nontrivial) divisorial valuations $v$.
(It is klt when $B$ is further effective.)

A pair $(X,B)$ is \emph{log smooth} if $X$ is smooth and $B$ has simple normal 
crossing support. 
A \emph{log resolution} of $(X,B)$ is a projective birational morphism 
$f\colon X'\to X$, with $X'$ smooth, such that $\Exc(f)+f^{-1}_*(B)$ has 
simple normal crossing support. 
In this case, there is a unique snc divisor $B'$ on $X'$ such that $f_*B'=B$
and $K_{(X',B')}=f^*K_{(X,B)}$. In particular the pair $(X',B')$ is log smooth.
The pair $(X,B)$ is subklt iff $(X',B')$ is subklt, and the latter is equivalent to $B'$
having coefficients $<1$.

A smooth metric $\p$ on $K_{(X,B)}$ canonically defines a smooth positive measure
$\mu_\p$ on $X_{\reg}\setminus B$ as follows. Let $\phi_B$ be the canonical singular 
metric on $\cO_{X_{\reg}}(B)$, with curvature current given by $[B]$. 
Then $\p-\phi_B$ is a smooth metric on 
$K_{X_{\reg}\setminus B}$, and hence induces a smooth positive measure
\begin{equation*}
  \mu_\p:=e^{2(\p-\phi_B)}
\end{equation*}
on $X_{\reg}\setminus B$.
The fact that $(X,B)$ is subklt means precisely that the total mass of $\mu_\p$
is finite. Thus we can view $\mu_\p$ as a finite positive measure on $X$ that is smooth
on $X_{\reg}\setminus B$ and gives no mass to $B$ or $X_{\sing}$. 
%
%
\subsection{Archimedean functionals}\label{S203}
Let $X$ be a normal complex projective variety of dimension $n$.
Fix a big and nef $\Q$-line bundle $L$ on $X$ and set $V:=(L^n)>0$.
For a smooth metric $\phi$ on $L$, set $\MA(\phi)=V^{-1}(dd^c\phi)^n$.

Fix a smooth positive reference metric $\phi_\re$ on $L$
The energy functionals $E$, $I$ and $J$ are defined on smooth metrics on $L$
exactly as in~\eqref{equ:E},~\eqref{equ:I} and~\eqref{equ:J}, respectively;
they are normalized by $E(\phi_\re)=I(\phi_\re)=J(\phi_\re)=0$.
The functionals $I$ and $J$ are translation invariant, whereas 
$E(\phi+c)=E(\phi)+c$. All three functionals are pullback invariant in the following 
sense. Let $q\colon X'\to X$ be a birational morphism, with $X'$ normal and projective, 
and set $L':=q^*L$.
For any smooth metric $\phi$ on $L$, we have $E(\phi')=E(\phi)$, 
$I(\phi')=I(\phi)$ and $J(\phi')=J(\phi)$, where $\phi'=q^*\phi$ and where
the functionals are computed with
respect to the reference metric $\phi'_\re:=q^*\phi_\re$. 

Now consider a boundary $B$ on $X$. Set $\bar{S}_B:=-nV^{-1}(K_{(X,B)}\cdot L^{n-1})$
and fix a smooth reference
metric $\p_\re$ on $K_{(X,B)}$. When $X$ is smooth and $B=0$, we could pick
$\p_\re=\frac12\log\MA(\phi_\re)$, 
but in general, there seems to be no canonical way to get $\p_\re$ from $\phi_\re$.

The analogue of the Ricci energy $R$ is defined on smooth metrics $\phi$ on $L$ by
\begin{equation*}
  R_B(\phi)
  :=\sum_{j=0}^{n-1}\frac1V
  \int_{X_{\reg}}(\phi-\phi_\re)dd^c\p_\re\wedge(dd^c\phi)^j\wedge(dd^c\phi_\re)^{n-1-j}.
\end{equation*}
It satisfies $R_B(\phi+c)=R_B(\phi)-\bar{S}_Bc$ and is 
pullback invariant in the following sense.
Suppose $q\colon X'\to X$ is a birational morphism, with $X'$
projective normal, and define $B'$ by $q_*B'=B$ and $q^*K_{(X,B)}=K_{(X',B')}$.
Set $\phi'_\re=q^*\phi_\re$ and $\p'_\re:=q^*\p_\re$.
Then $R_B(\phi)=R_{B'}(\phi')$, where $\phi'=q^*\phi$.

Now assume $(X,B)$ is subklt and 
let $\mu_\re=\mu_{\p_\re}$ be the finite positive measure defined in~\S\ref{S202}.  It is smooth and positive on $X_{\re}\setminus B$, and may be assumed to have mass $1$, after adding a constant to $\p_\re$. 
For a smooth semipositive metric $\phi$ on $L$, set
\begin{equation*}
  H_B(\phi)
  :=\frac12\int_{X_{\reg}}\log\frac{\MA(\phi)}{\mu_\re}\MA(\phi)
  =\frac12\int_{X_{\reg}}\log\frac{\MA(\phi)}{e^{2(\p_\re-\phi_B)}}\MA(\phi).
\end{equation*}
We may have $H_B(\phi_\re)\ne0$. However,
$H_B$ is bounded from below and translation invariant. 
It is also pullback invariant in the sense above, with reference measure 
$\mu'_\re=\mu_{\p'_\re}$ on $X'$.
\begin{lem}\label{L201}
  If $\phi$ is a smooth semipositive metric on $L$, then 
  $H_B(\phi)<+\infty$.
\end{lem}
\begin{proof}
  By pullback invariance we may assume that $(X,B)$ is log smooth.
  In this case $\MA(\phi)$ and $\mu_\re$ are smooth measures on $X$
  that are strictly positive on $X_{\reg}$. Consider any point $\xi\in B$ and pick
  local coordinates $(z_1,\dots,z_n)$ at $\xi$ such that the irreducible components
  of $B$ are given by $\{z_i=0\}$, $0\le i\le p$. Fix a volume form $dV$ near $\xi$.
  Then $\mu_\re=g\prod_{i=0}^p|z_i|^{2a_i}dV$, and $\MA(\phi)=hdV$, with 
  $a_i>-1$, $g>0$ and $h\ge0$ smooth. If $f=h\log(\frac{h}{g}\prod_{i=0}^p|z_i|^{-2a_i})$, then 
  $f$ is locally integrable with respect to $dV$. This completes the proof.
\end{proof}
As in~\S\ref{S101} we define the Mabuchi functional on semipositive smooth metrics by
\begin{equation*}
  M_B:=H_B+R_B+\bar{S}_BE.
\end{equation*}
Then $M_B$ is translation invariant and pullback invariant in the sense above.
At least formally, the critical points of $M_B$ satisfy
\begin{equation*}
  n(\Ric(dd^c\phi)-[B])\wedge(dd^c\phi)^{n-1}
  =\bar{S}_B(dd^c\phi)^n
\end{equation*}
and should be conical cscK metrics, see~\cite{Li14}.

\medskip
Finally consider the (weak) \emph{log Fano case}, in which $L:=-K_{(X,B)}$ is big and nef.
The Ding functional is then defined on smooth metrics as $D_B=L_B-E$, with 
\begin{equation*}
  L_B(\phi):=-\frac12\log\int_{X_{\reg}}e^{-2(\phi+\phi_B)}.
\end{equation*}
If we use $\p_\re=-\phi_\re$, then 
the formula for the Mabuchi functional simplifies to
\begin{equation*}
  M_B(\phi)=H_B(\phi)-(E(\phi)-\int_{X_\mathrm{\reg}}(\phi-\phi_\re)\MA(\phi)).
\end{equation*}
We have $D_B\le M_B$ on smooth semipositive metrics.
%
%
\subsection{Non-Archimedean functionals}
The extensions of the non-Archimedean functionals in~\S\ref{S111} to the
logarithmic setting were studied in~\cite[\S7]{BHJ1}. Let us briefly review them.

Consider a normal complex projective variety $X$ and a big and nef $\Q$-line bundle
$L$ on $X$.
Let $\phi$ be a non-Archimedean metric on $L$, represented by a normal
test configuration $(\cX,\cL)$ for $(X,L)$, that we assume dominates 
$(X\times\C,L\times\C)$ via $\rho\colon\cX\to X\times\C$.
The formulas in~\S\ref{S111} for $E^\NA(\phi)$, $I^\NA(\phi)$ and $J^\NA(\phi)$ 
are still valid. 

Given a boundary $B$ on $X$ we set
\begin{align*}
  R_B^{\NA}(\phi)
  :&=V^{-1}(\psi_\triv\cdot\phi^n)\\
  &=V^{-1}\left(\rho^*K^\lo_{(X\times\P^1,B\times\P^1)/\P^1}\cdot\bar\cL^n\right).
\end{align*}

Now assume $(X,B)$ is subklt and let 
$\cB$ (resp.\ $\bar\cB$) be the (component wise) Zariski closure of $B\times\C^*$
in $\cX$ (resp.\ $\bar\cX$).
Then 
\begin{align*}
  H^{\NA}_B(\phi)
  :&=\int_{\Xdiv}A_{(X,B)}(v)\MA^{\NA}(\phi)\\
   &=V^{-1}\left(K^\lo_{(\bar\cX,\bar\cB)/\P^1}\cdot\bar\cL^n\right)
     -V^{-1}\left(\rho^*K^\lo_{(X\times\P^1,B\times\P^1)/\P^1}\cdot\bar\cL^n\right).
\end{align*}
and
\begin{align*}
  M_B^{\NA}(\phi)
  :&=H_B^{\NA}(\phi)+R_B^{\NA}(\phi)+\bar S_B E^{\NA}(\phi)\\
   &=\frac1V\left(K^\lo_{(\bar\cX,\bar\cB)/\P^1}\cdot\bar\cL^n\right) 
     +\frac{\bar S_B}{(n+1)V}\left(\bar\cL^{n+1}\right). 
\end{align*}
While the definitions of $H_B^{\NA}(\phi)$ and $M_B^{\NA}(\phi)$
make sense for arbitrary non-Archimedean metrics $\phi$,
we will usually assume that $\phi$ is semipositive.

All the functionals above have the same invariance properties as their
Archimedean cousins. They are also homogeneous in the sense of Definition~\ref{D201}.

\medskip
Finally, when $(X,B)$ is weakly log Fano, so that $(X,B)$ is subklt and 
$L:=-K_{(X,B)}$ is big and nef,
the non-Archimedean Ding functional is defined by
\begin{equation*}
  D_B^\NA(\phi)=L_B^\NA(\phi)-E^\NA(\phi),
\end{equation*}
where 
\begin{equation*}
  L_B^\NA(\phi)
  =\inf_v\left(A_{(X,B)}(v)+(\phi-\phi_{\triv})(v)\right),
\end{equation*}
the infimum taken over all valuations $v$ on $X$ that are divisorial
or trivial.

The Ding functional $D_B^\NA$ is translation invariant
and pullback invariant.
The formula for the Mabuchi functional simplifies in the log Fano case to
\begin{equation*}
  M^\NA_B(\phi)=H^\NA_B(\phi)-(E^\NA(\phi)-\int_{\Xdiv}(\phi-\phi_\re)\MA^\NA(\phi)).
\end{equation*}
We have $D_B^\NA\le\min\{M_B^\NA, J^\NA\}$ on semipositive metrics, 
see~Propositions~7.28 and~7.32 in~\cite{BHJ1}. 
%
%
\subsection{Asymptotics}
The following result generalizes Theorem~\ref{thm:asymfunc}
and shows that if $F$ is one of the functionals $E$, $I$, $J$, $H_B$, $R_B$ or $M_B$
on $\cH$, then $F$ admits a non-Archimedean limit on $\cH^{\NA}$ given by $F^\NA$.
For future reference, we state the result in detail.
\begin{thm}\label{T201}
  Let $X$ be a normal projective variety, $L$ a big and nef $\Q$-line bundle on $X$,
  and $(\cX,\cL)$ a test configuration for $(X,L)$ inducing a non-Archimedean
  metric $\phi^\NA$ on $L$. 
  Further, let $\Phi$ be a smooth, $S^1$-invariant metric on $\cL$ near $\cX_0$,
  inducing a smooth ray $(\phi^s)_{s>s_0}$ of metrics on $L$. 
  Fix a smooth reference metric $\phi_\re$ on $L$.
  Then 
  \begin{equation}\label{e202}
    \lim_{s\to+\infty}\frac{F(\phi^s)}{s}=F^\NA(\phi^\NA),
  \end{equation}
  where $F$ is any of the functionals $E$, $I$, $J$.
  
  Further, if $B$ is a boundary on $X$ and $\p_\re$ is a smooth reference metric 
  on $K_{(X,B)}$, then~\eqref{e202} also holds for $F=R_B$.
  Finally, if $(X,B)$ is subklt and $\Phi$ is semipositive, then~\eqref{e202} 
  holds for $F=H_B$ and $F=M_B$.
\end{thm}
In addition, Berman proved that in the log Fano case,
the Ding functional $D_B$ admits $D_B^\NA$ as non-Archimedean limit.
Indeed, the following result follows from Proposition~3.8 and~\S4.3 in~\cite{Berm16}.
\begin{thm}\label{T202}
  Let $(X,B)$ be a subklt pair with $L:=-K_{(X,B)}$ big and nef,
  $(\cX,\cL)$ a test configuration for $(X,L)$ inducing a non-Archimedean
  metric $\phi^\NA$ on $L$, and 
  $\Phi$ a semipositive smooth, $S^1$-invariant metric on $\cL$ near $\cX_0$,
  inducing a smooth ray $(\phi^s)_{s>s_0}$ of semipositive metrics on $L$. 
  Then  $\lim_{s\to+\infty}\frac1sD_B(\phi^s)=D_B^\NA(\phi^\NA)$.
\end{thm}
In fact, it is enough to assume $\Phi$ is semipositive and locally bounded in 
Theorem~\ref{T202}.
\begin{rmk}
  Theorems~\ref{T201} and~\ref{T202}
  remain true even when $\Phi$ is not $S^1$-invariant, in the following 
  sense. For $\tau\in\D^*$, let $\phi_\tau$ be the metric on $L$ defined as the 
  pullback of $\Phi|_{\cX_\tau}$ under the $\C^*$-action.
  Then we have $\lim_{\tau\to0}(\log|\tau|^{-1})^{-1}F(\phi_\tau)=F^\NA(\phi^\NA)$. 
\end{rmk}
%
%
\subsection{Proof of Theorem~\ref{T201}}
By pullback invariance, we may assume that $X$ is smooth. After further pullback, we may 
also assume that $\cX$ is smooth and dominates $X\times\C$. In this case, the 
asymptotic formulas for $E$, $I$ and $J$ follow immediately from Lemma~\ref{lem:deligne}. 

When considering the remaining functionals, we may similarly, by pullback invariance, 
assume that the pair $(X,B)$ is log smooth.
The asymptotic formula for $R_B$ now follows from Lemma~\ref{lem:deligne}
since we can express $R_B(\phi)$ in terms of Deligne pairings:
\begin{equation*}
  R_B(\phi)
  =\langle\p_\re,\phi^n\rangle_X
    -\langle\p_\re,\phi_\re^n\rangle_X,
\end{equation*}
whereas the non-Archimedean counterpart is given by the intersection number
\begin{equation*}
  R_B^{\NA}(\phi)
  =V^{-1}\left(\rho^*K^\lo_{(X\times\P^1,B\times\P^1)/\P^1}\cdot\bar\cL^n\right)_{\bar\cX}.
\end{equation*}

Finally we consider the functionals $H_B$ and $M_B$. Thus assume $(X,B)$ 
is log smooth and subklt. We may further assume that 
the divisor $\cX_0+\cB$ has simple normal crossing support, 
where $\cB$ is the (component-wise) Zariski closure of the pullback of 
$B\times\C^*$ in $\cX$.

As in~\S\ref{S112} it suffices to prove the asymptotic formula for the functional
$H_B+R_B$. To this end, we express $H_B$ in terms of Deligne pairings.
Write $B=\sum_i c_iB_i$, where $B_i$, $i\in I$, 
are the irreducible components of $B$ and $c_i\in\Q$.
Fix a smooth metric $\p_i$ on $\cO_X(B_i)$ for $i\in I$. Then 
$\p_B:=\sum_ic_i\p_i$ is a smooth metric on $\cO_X(B)$, and 
it follows from~\eqref{equ:metrind} that 
\begin{align*}
  VH_B(\phi)
  &=\frac12\int_X\log\frac{\MA(\phi)}{e^{2(\p_\re-\p_B)}}(dd^c\phi)^n
    +\sum_{i\in I}c_i\int_X\log|\sigma_i|_{\p_i}(dd^c\phi)^n\\
  &=\langle\tfrac12\log\MA(\phi),\phi^n\rangle_X
    -\langle\p_\re,\phi^n\rangle_X
    +\langle\p_B,\phi^n\rangle_X
    +\sum_{i\in I}c_i\left(
    \langle\phi^n\rangle_{B_i}-\langle\p_i,\phi^n\rangle_X
    \right)\\
  &=\langle\tfrac12\log\MA(\phi),\phi^n\rangle_X
    -\langle\p_\re,\phi^n\rangle_X
  +\sum_{i\in I}c_i\langle\phi^n\rangle_{B_i},
\end{align*} 
for any smooth semipositive metric $\phi$ on $L$. This implies
\begin{align*}
  V(H_B(\phi)+R_B(\phi))
  &=\langle\tfrac12\log\MA(\phi),\phi^n\rangle_X
    -\langle\p_\re,\phi_\re^n\rangle_X
    +\sum_{i\in I}c_i\langle\phi^n\rangle_{B_i}\\
  &=V(H(\phi)+R(\phi))
    +n\sum_{i\in I}c_i(L^{n-1}\cdot B_i)E(\phi|_{B_i})+O(1).
\end{align*} 

On the non-Archimedean side, we have 
\begin{align*}
  V(H_B^{\NA}(\phi^\NA)+R_B^{\NA}(\phi^\NA))
  &=\left(K^{\log}_{(\bar\cX,\bar\cB)/\P^1}\cdot\bar\cL^n\right)_{\bar\cX}\\
  &=\left(K^{\log}_{\bar\cX/\P^1}\cdot\bar\cL^n\right)_{\bar\cX}
  +\left(\bar\cB\cdot\bar\cL^n\right)_{\bar\cX}\\
  &=V(H^\NA(\phi^\NA)+R^\NA(\phi^\NA))
    +\sum_{i\in I}c_i\left(\bar\cL|_{\bar\cB_i}^n\right)_{\bar\cB_i}\\
  &=V(H^\NA(\phi^\NA)+R^\NA(\phi^\NA))
    +n\sum_{i\in I}c_i(L^{n-1}\cdot B_i)E^\NA(\phi_i^\NA),
\end{align*}
where $\phi_i^\NA$ is the non-Archimedean metric on $L|_{B_i}$ 
represented by $\cL|_{\cB_i}$.

It now follows from Theorem~\ref{thm:asymfunc} that\footnote{While Theorem~\ref{thm:asymfunc} is stated in the case when $L$ and $\cL$ are ample and $\Phi$ is positive, the proof extends to the weaker positivity assumptions used here.}
\begin{equation*}
  \lim_{s\to\infty}\frac1s(H(\phi^s)+R(\phi^s))
  =H^\NA(\phi^\NA)+R(\phi^\NA),
\end{equation*}
Applying Theorem~\ref{thm:asymfunc} on $B_i$ and $\cB_i$, we also get
$\lim_{s\to\infty}\frac1sE(\phi_i^s)=E^\NA(\phi_i^\NA)$.
Thus
\begin{equation*}
  \lim_{s\to\infty}\frac1s(H_B(\phi^s)+R_B(\phi^s))
  =H_B^\NA(\phi^\NA)+R_B(\phi^\NA),
\end{equation*}
which completes the proof of Theorem~\ref{T201}.
%
%
\subsection{Coercivity and uniform K-stability}
Let us finally extend Corollary~B to the logarithmic setting. Consider a pair $(X,B)$
and a big and nef line bundle $L$ on $X$. The Donaldson-Futaki invariant of a 
normal test configuration $(\cX,\cL)$ for $(X,L)$ is given by 
\begin{align*}
  \DF_B(\cX,\cL)
  :&=\frac1{V}(K_{(\bar\cX.\bar\cB)/\P^1}\cdot\bar\cL^n)
     +\bar S_B\frac{(\bar\cL^{n+1})}{(n+1)V}\\
   &=M_B^\NA(\phi)+\frac{1}{V}\left((\cX_0-\cX_{0,\mathrm{red}})\cdot\cL^n\right),
\end{align*}
where $\phi$ is the non-Archimedean metric on $L$ represented by $\phi$.
Now assume $L$ is ample. We then define $(X,B);L)$ to be \emph{uniformly $K$-stable}
if the following two equivalent conditions hold:
\begin{itemize}
\item[(i)]  
  there exists $\d>0$ such that 
  $M_B^{\NA}(\phi)\ge\delta J^{\NA}(\phi)$ for every $\phi\in\cH^{\NA}(L)$;
\item[(ii)]
  there exists $\d>0$ such that 
  $\DF_B(\cX, \cL) \geq \delta J^{\NA}(\cX, \cL)$ for any normal 
  ample test configuration $(\cX, \cL)$. 
\end{itemize} 
The equivalence between the two conditions is proved in~\cite[Proposition~8.2]{BHJ1}.
\begin{cor}\label{C201}
  Let $(X,B)$ be a subklt pair and $L$ an ample line bundle on $X$. 
  Suppose that the Mabuchi functional is coercive in the sense that 
  there exist positive constants $\delta$ and $C$ such that 
  $M_B(\phi)\ge\delta J(\phi)-C$ for every positive smooth metric $\phi$ on $L$.
  Then $((X,B);L)$ is uniformly K-stable; more precisely
  $\DF_B(\cX,\cL)\ge M_B(\phi)\ge \delta J^{\NA}(\phi)$ for every positive non-Archimedean
  metric on $L$, where $(\cX,\cL)$ is the unique normal ample representative of~$\phi$.
\end{cor}
%
%
%
%
\section{Uniform K-stability and CM-stability}\label{sec:CM}
From now on, $X$ is smooth.
In this section we explore the relationship between uniform
K-stability and (asymptotic) CM-stability.
In particular we prove Theorem~C, Corollary~D and Corollary~E.
%
%
\subsection{Functions with log norm singularities}
In this section, $G$ denotes a reductive complex algebraic group. 
\begin{defi}\label{defi:norms} 
  We say that a function $f:G\to\R$ has \emph{log norm singularities} if there exist finitely many rational numbers $a_i$, finite dimensional complex vector spaces $V_i$ endowed with an algebraic $G$-action and non-zero vectors $v_i\in V_i$ such that
  $$
  f(g)=\sum_i a_i\log\|g\cdot v_i\|+O(1)
  $$
  for some choice of norms on the $V_i$'s. 
\end{defi}

\begin{rmk} By the equivalence of norms on a finite dimensional vector
  space, the description of $f$ is independent of the choice of norms
  on the $V_i$. In particular, given a maximal compact subgroup $K$ of $G$, the norms may be assumed to be $K$-invariant, so that $f$ descends to a function on the Riemannian symmetric space $G/K$.  
\end{rmk}

\begin{rmk}\label{rmk:lognorm}  
  Taking appropriate tensor products, is is easy to see that every
  function $f$ on $G$ with log norm singularities may be written as
  \begin{equation}\label{equ:flog}
    f(g)=a\left(\log\|g\cdot v\|-\log\|g\cdot w\|\right)+O(1),
  \end{equation}
  where $a\in\Q_{>0}$ and $v$, $w$ 
  are vectors in a normed vector space $V$ endowed with a $G$-action.
\end{rmk}
An algebraic group homomorphism $\la:\C^*\to G$ is called a \emph{one-parameter subgroup} (\emph{$1$-PS} for short). \corr{The following generalization of the Kempf-Ness/Hilbert-Mumford
criterion is closely related to~\cite{Paul13}. Our argument, which is based on Mumford's original proof of the Hilbert-Mumford criterion \cite[\S 2.1]{GIT}, fixes in particular the proof of \cite[Theorem 4.2]{Paul13}, as well as an incorrect argument provided in a previous version of the present paper. }

\begin{thm}\label{thm:KN} Let $f$ be a function on $G$ with log norm singularities. 
\begin{itemize}
\item[(i)] For each $1$-PS $\la\colon\C^*\to G$, there exists $f^{\NA}(\la)\in\Q$ such that 
$$
(f\circ\la)(\tau)=f^{\NA}(\la)\log|\tau|^{-1}+O(1)
$$
for $|\tau|\le 1$. 
\corr{\item[(ii)] The function $f$ is bounded below on $G$ iff $f^{\NA}(\la)\ge 0$ for all $1$-PS $\la$.}
\end{itemize}
\end{thm}
The chosen notation stems from the fact that $f^{\NA}$ induces a function on the (conical) Tits building of $G$, \ie the non-Archimedean analogue of $G/K$ (compare~\cite[\S 2.2]{GIT}). 

\medskip

By Remark~\ref{rmk:lognorm}  we may and do assume that $f$ is of the form 
  \begin{equation*}
    f(g):=\log\|g\cdot v\|-\log\|g\cdot w\|,
  \end{equation*}
  where $v$, $w$ are nonzero vectors in a finite dimensional normed vector space $V$
  equipped with a linear $G$-action. In that case, the following variant of the Kempf-Ness criterion, observed in \cite[Proposition 4]{Paul13}, translates Theorem~\ref{thm:KN} into an algebro-geometric statement. 
  
\begin{lem}\label{lem:KN} The function $f(g)=\log\|g\cdot v\|-\log\|g\cdot w\|$ is bounded below on $G$ if and only if the Zariski closure of the orbit of $[v,w]\in\P(V\oplus W)$ does not intersect the subspace $\P(\{0\}\oplus W)$. 
\end{lem}
\begin{proof} As with any algebraic group action, the orbit $G\cdot [v,w]$ is a complex algebraic subvariety of $\P(V\oplus W)$, \ie a locally closed subset in the Zariski topology. Its Zariski closure therefore coincides with its closure in the Euclidean topology, and the argument is then elementary. Indeed, assume $f(g_i)\to-\infty$ for some sequence $g_i\in G$, \ie $\|g_i\cdot v\|=o\left(\|g_i\cdot w\|\right)$. After passing to a subsequence, $\tilde w_i:=(g_i\cdot w)/\|g_i\cdot w\|$ converges (in the Euclidian topology) to a nonzero vector in $W$, while $\tilde v_i:=(g_i\cdot v)/\|g_i\cdot w\|$ tends to $0$; hence $g_i\cdot [v,w]\in G\cdot [v,w]$ converges to $[0,\tilde w]\in\P(\{0\}\oplus W)$. Conversely, if $g_i\cdot [v,w]\to [0,\tilde w]$ for some sequence $g_i\in G$ and nonzero $\tilde w\in W$, then $c_i(g_i\cdot v)\to 0$ in $V$ and $c_i(g_i\cdot w)\to\tilde w$ in $W$ with $c_i\in\C^*$, and hence $f(g_i)=\log\|c_i(g_i\cdot v)\|-\log\|c_i(g_i\cdot w)\|\to-\infty$. 
\end{proof}

\corr{The key ingredient in the proof of Theorem~\ref{thm:KN} is the following algebro-geometric result, which will be obtained as a consequence of the Iwahori decomposition theorem, very much as in \cite{GIT}. }

\corr{\begin{thm}\label{thm:spec} Let $G$ be a complex reductive group with a linear action on a finite dimensional complex vector space $U$. If the (Zariski) closure of the $G$-orbit of a point $x\in \P(U)$ meets a $G$-invariant Zariski closed subset $Z\subset\P(U)$, then some $z\in Z\cap\overline{G\cdot x}$ can be reached by a $1$-PS $\la$ of $G$, \ie $\lim_{\tau\to 0}\la(\tau)\cdot x=z$.
\end{thm}}

\begin{rmk} As explained in \cite[\S 5]{Don10}, it is however not true in general that \emph{any} $z\in Z\cap\overline{G\cdot x}$ can be reached by a $1$-PS $\la$, unless the stabilizer of $z$ in $G$ is reductive. 
\end{rmk}

Introduce the formal power series ring $R=\C\cro{t}$ and its fraction field $K:=\C\lau{t}$, and let $X$ be a complex algebraic variety. Viewed as a $\C$-scheme, $X$ is separated, and the set of $R$-points $X(R)$, \ie morphisms $\g:\Spec R\to X$ over $\spec\C$, thus injects into $X(K)$. Further, each $\g\in X(R)$ admits a reduction $\tilde\g\in X(\C)$. If $X$ is proper (\ie $X(\C)$ is compact), the valuative criterion yields $X(R)=X(K)$, which means that any 'meromorphic arc' $\g:\Spec K\to X$ uniquely extends across the closed point of $\Spec R$, whose image is $\tilde\g$. In case $X=\P(U)$ for a complex vector space $U$, this becomes very concrete: for each $\g\in X(K)=\P(U_K)$, there exists $u\in U_R$, unique up to multiplication by a unit of $R$, such that $\g=[u]$ and $\tilde u\ne 0$ in $U$, and we then have $\tilde\g=[\tilde u]$. 

The following valuative criterion was used without precise reference in Mumford's proof of the Hilbert-Mumford criterion \cite[p.54]{GIT}. We provide here some details (see \cite[\S 4]{Ant} for a closely related discussion). 

\begin{lem}\label{lem:val} Let $\phi:Y\to X$ be a morphism between complex algebraic varieties, and let $x\in X(\C)$ be a closed point. Then $x$ belongs to the (Zariski) closure of the image $\phi(Y)$ if and only if there exists $\g\in Y(K)$ with $\phi(\g)\in X(R)$ and $\widetilde{\phi(\g)}=x$. 
\end{lem}
\begin{proof} The condition is clearly sufficient. Assume conversely that $x$ is in the Zariski closure of $\phi(Y)$. Replacing $X$ with the closure of $\phi(Y)$, we may assume that $\phi$ is dominant. By Chevalley's theorem, $\phi(Y)$ is constructible, \ie a finite union of locally closed subsets; being dense in $X$, it thus contains a non-empty open subset $U\subset X$. Using for instance Noether normalization, it is easy to construct a closed point $p\in C$ on a smooth algebraic curve and a morphism $f:C\to X$ with $f(p)=x$ and $f^{-1}(U)$ non-empty \cite[Lemma 7.2.1]{Kem}. It follows that the restriction of the induced morphism $\spec\cO_{C,p}\to X$ to the generic point lifts to $Y$, and passing to the formal completion of $C$ at $p$ yields the result. 
\end{proof}

\begin{proof}[Proof of Theorem~\ref{thm:spec}] The action of $G$ on $X:=\P(U)$, being algebraic, induces an action of the group $G(K)$ on the set $X(K)$. Since $K$ is an extension of $\C$, the closed point $x\in X$ can be viewed as an element of $X(K)$, and our goal is to find a point $\la\in G(K)$ corresponding to a one-parameter subgroup of $G$ such that the reduction of $\la\cdot x\in X(K)$ belong to $Z$. 

\corr{Given any $1$-PS $\la\in G(K)$ and $\xi\in X(K)$, we first claim that the reduction of $\la\cdot\xi \in X(K)$ only depends on $\tilde\xi \in X(\C)$. Indeed, denote by $U=\bigoplus_{m\in\Z} U_m$ the weight decomposition with respect to $\la$. As mentioned above, there exists $u\in U_R$, unique up to a unit in $R$, such that $\xi=[u]$ and $\tu\ne 0$. The reduction of 
$$
\la\cdot \xi=\left[\sum_m t^m u_m\right]
$$
is equal to $[\widetilde{u_p}]$ with $p:=\min\{m\mid\widetilde{u_m}\ne 0\}$, and hence only depends on $\tilde\xi=[\tilde u]$. }

\blue{\begin{exam} This claim is incorrect, as illustrated by the following counterexample kindly communicated to us by Yan Li. Consider $\xi:=[1:0]$, $\xi':=[1:t]$ in $\P^1(K)$ and the $1$-PS $\la:=\mathrm{diag}(t^2,t^{-2})$. Then $\tilde\xi=\tilde\xi'=[1:0]\in\P^1(\C)$, but $\la\cdot\xi=[t^2:0]=[1:0]$, $\la\cdot\xi'=[t^2:t^{-1}]=[t^3:1]$, and hence $\widetilde{\la\cdot \xi}=[1:0]\ne [0:1]=\widetilde{\la\cdot \xi'}$. 
\end{exam}}

Now let $\phi:G\to X$ be the orbit morphism $\phi(g)=g\cdot x$. By assumption, $\phi(G)$ contains a closed point $z\in Z$ in its Zariski closure, and Lemma~\ref{lem:val} thus implies the existence of $\g\in G(K)$ such that the reduction of $\phi(\g)=\g\cdot x\in X(K)$ is equal to $z$. 

By Iwahori's theorem (cf.~\cite[p.52]{GIT}), we can find a decomposition $\g=\a\la\b$ in $G(K)$ with $\a,\b\in G(R)$ and $\la\in G(K)$ induced by a $1$-PS. By $G$-invariance of $Z$, the reduction of $(\la\b)\cdot x$ belongs to $Z$. After replacing $\la$ with $\tilde\b\la\tilde\b^{-1}$ and $\b$ with $\tilde\b^{-1}\b$, we may assume that $\tilde\b=e\in G(\C)$. As a result, $\tilde x=\widetilde{\b\cdot x}$, \corr{and the above claim implies that $\widetilde{\la\cdot x}=\widetilde{(\la\b)\cdot x}$ belongs to $Z$.}
\end{proof}

\begin{proof}[Proof of Theorem~\ref{thm:KN}]
  
    (i) Let $\la\colon\C^*\to G$ be a $1$-parameter subgroup, and denote by $V=\bigoplus_{m\in\Z} V_m$ the corresponding weight decompositon. For $\tau\in\C^*$, we then have 
  \begin{equation*}
    \la(\tau)\cdot v=\sum_m \tau^m v_m, 
  \end{equation*}
  and hence 
  \begin{equation*}
    \log\|\la(\tau)\cdot v\|=\max_{v_m\ne 0}\left(m\log|\tau|+\log\|v_m\|\right)+O(1)
    =-\left(\min_{v_m\ne 0}m\right)\log|\tau|^{-1}+O(1)
  \end{equation*}
  for $|\tau|\le 1$. This proves (i) with
  with $f^\NA(\lambda)=\min\{m\mid w_m\ne 0\}-\min\{m\mid v_m\ne 0\}$. 

  (ii) By (i), $f^\NA(\la)\ge 0$ for all $1$-PS $\la$ if and only if $f\circ\la$ is bounded below on $\C^*$ for all $\la$. 
 By Lemma~\ref{lem:KN}, $f\circ\la$ is bounded below on $\C^*$ iff $\lim_{\tau\to 0}\la(\tau)\cdot[v,w]$ does not belong to the $G$-invariant Zariski closed subset $Z:=\P(\{0\}\oplus W)$, while $f$ is bounded below on $G$ iff $Z\cap\overline{G\cdot [v,w]}=\emptyset$. \corr{The equivalence now follows from Theorem~\ref{thm:spec}.}
 \end{proof}
%
%
\subsection{\corr{Proof of Theorem~C and Corollaries~D and~E}}
Replacing $L$ with $mL$, we may assume for notational simplicity that $m=1$. 
Set $N:=h^0(L)$ and $G:=\SL(N,\C)$, so that each $\sigma\in G$ defines
a Fubini-Study type metric $\phi_\sigma$ on $L$. 
Note that $M-\d J$ is bounded below on $\cH_1\simeq\GL(N,\C)/\U(N)$ 
iff $M(\phi_\sigma)-\d J(\phi_\sigma)$ bounded below for $\sigma\in G$,
by translation invariance of $M$ and $J$. 

The key ingredient is the following result of S.~Paul~\cite{Paul12} (see also \cite{Kap})
\begin{thm}\label{thm:norms}
  The functionals $E$, $J$ and $M$ all have log norm singularities on $G$.  
\end{thm}
\corr{Granted this result we can deduce Theorem~C.}
The equivalence of~(ii) and~(iii) follows from the same argument as 
Proposition~8.2 in~\cite{BHJ1}, so it suffices to show that~(i) 
and~(iii) are equivalent.
By Theorem~\ref{thm:norms}, the function
$f(\sigma):=M(\phi_\sigma)-\d J(\phi_\sigma)$ on $G$ 
has log norm singularities.  
\corr{By Theorem~\ref{thm:KN}, it is thus bounded below iff
\begin{equation*}
  \lim_{s\to+\infty}\frac{(f\circ\la)(e^{-s})}{s}\ge 0
\end{equation*}
for each $1$-parameter subgroup $\la\colon\C^*\to G$.}
\corr{We obtain the desired result} since by Theorem~B, this limit is equal to 
$M^{\NA}(\phi_\la)-\d J^{\NA}(\phi_\la)$,
where $\phi_\la\in\cH^\NA$ is the non-Archimedean metric on $L$
defined by $\la$.

\corr{Corollary~D follows since every ample test configuration of $(X,L)$ 
is induced by a 1-PS, see~\S\ref{S109}. The first assertion of
Corollary~E follows immediately, and the fact that the reduced
automorphism group of $(X,L)$ is finite is a consequence 
of~\cite[Corollary 1.1]{Paul13}.}

\begin{proof}[Proof of Theorem~\ref{thm:norms}]
  Recall from~\cite{Paul12} that to the linearly normal 
  embedding $X\hookrightarrow\P H^0(X,L)^*\simeq\P^{N-1}$ 
  are associated the $X$-resultant $R$, \ie the Chow coordinate 
  of $X$, and the $X$-hyperdiscriminant $\D$, which cuts out the dual variety of 
  $$
  X\times\P^{n-1}\hookrightarrow\P^{N-1}\times\P^{n-1}\hookrightarrow\P^{Nn-1}, 
  $$
  the second arrow being the Segre embedding. 

In our notation, we then have $\deg R=V(n+1)$ and $\deg\D=V\left(n(n+1)-\bar S\right)$~\cite[Proposition 5.7]{Paul12}, and~\cite[Theorem A]{Paul12} becomes
\begin{equation}\label{equ:MP}
M(\phi_\sigma)=V^{-1}\log\|\sigma\cdot\D\|-V^{-1}\frac{\deg\D}{\deg R}\log\|\sigma\cdot R\|+O(1), 
\end{equation}
which proves the assertion for $M(\phi_\sigma)$. 

We next consider 
$$
J(\phi_\sigma)=\int_X(\phi_\sigma-\phi_{\mathrm{ref}})\MA(\phi_{\mathrm{ref}})-E(\phi_\sigma).
$$
On the one hand, by~\cite[Theorem 1]{Paul04} (or~\cite[Theorem 1.6, Theorem 3.6]{Zha96}) we have 
\begin{equation}\label{equ:EP}
E(\phi_\sigma)=\frac{1}{\deg R}\log\|\sigma\cdot R\|+O(1). 
\end{equation}
On the other hand, choosing any norm on the space of complex $N\times N$-matrices (in which $G$ of course embeds), it is observed in the proof of~\cite[Lemma 3.2]{Tian14} that 
$$
\int_X\left(\phi_\sigma-\phi_{\mathrm{ref}}\right)\MA(\phi_{\mathrm{ref}})=\log\|\sigma\|+O(1). 
$$
The assertion for $J(\phi_\sigma)$ follows. 
\end{proof}
%
%
\subsection{Discussion of~\cite{Tian14}}
The statement of~\cite[Lemma 3.1]{Tian14} sounds overoptimistic from the GIT point of view, as it would mean that CM-stability can be tested by only considering $1$-parameter subgroups of a fixed maximal torus $T$. 

At least, the proof is incorrect, the problem being the estimate~(3.1), which claims that $\phi_{\tau k}-\phi_\tau$ is uniformly bounded with respect to $\tau\in T$ and $k\in K$. As the next example shows, this is not even true for a fixed $k\in K$. 

\begin{exam} 
  Assume $(s_1,s_2)$ is a basis of $H^0(X,L)$, let $k\in U(2)$ be the
  unitary transformation exchanging $s_1$ and $s_2$,
  $\tau=(t,t^{-1})$, and pick a point $x$ with $s_1(x)=0$. Then 
$$
\phi_{\tau k}(x)-\phi_\tau (x)=4\log|\tau|
$$
is unbounded. 
\end{exam}

In any case, the methods here do not seem to be able to deduce
CM-stability from K-stability, because of the following fact
(cf.~\cite[p.39]{Li12}).
\begin{prop}\label{prop:bounded} For each polarized manifold $(X,L)$ and each $m$ large and divisible enough, there exists a non-trivial 1-PS $\la$ in $\GL(N_m,\C)$ such that $J$ and $M$ remain bounded on the corresponding Fubini-Study ray $\phi^s:=\phi_{\la(e^{-s})}$. 
\end{prop}
\begin{proof} As originally observed in~\cite{LX} (cf.~Proposition~\ref{prop:almost}), $(X,L)$ admits a non-trivial ample test configuration $(\cX,\cL)$ that is almost trivial, \ie with trivial normalization. As recalled in \S\ref{S109}, for each $m$ large and divisible enough, $(\cX,\cL)$ can be realized as the test configuration induced by a 1-PS $\la:\C^*\to\GL(N_m,\C)$, which is non-trivial since $(\cX,\cL)$ is. Since the normalization of $(\cX,\cL)$ is trivial, the associated non-Archimedean metric is of the form $\phi_\triv+c$ for some $c\in\Q$, and hence $M^\NA(\phi_\la)=J^\NA(\phi_\la)=0$. Since $M$ and $J$ have log norm singularities on $\GL(N_m,\C)$ by Theorem~\ref{thm:norms}, $M$ and $J$ are indeed bounded on $\phi^s$ by Theorem~\ref{thm:KN}. 
\end{proof}
%
%
%
%
\section{Remarks on the Yau-Tian-Donaldson conjecture}\label{S106}
As explained in the introduction, we will here give a simple argument,
following ideas of Tian, 
for the existence of a K\"ahler-Einstein metric on a Fano manifold
$X$, assuming $(X,-K_X)$ is uniformly K-stable and the 
partial $C^0$-estimates due to Sz\'ekelyhidi. 
%
%
\subsection{Partial $C^0$-estimates and the continuity method}
For the moment, consider an arbitrary polarized manifold $(X,L)$. 
For each $m$ such that $mL$ is very ample, we have a `Bergman kernel
approximation' map 
$P_m\colon\cH\to\cH_m$, defined by setting $P_m(\phi)$ to be the Fubini-Study metric induced by the $L^2$-scalar product on $H^0(X,mL)$ defined by $m\phi$.
\begin{defi} 
  A subset $A\subset\cH$ satisfies \emph{partial $C^0$-estimates at level $m$} 
  if there exists $C>0$ such that $|P_m(\phi)-\phi|\le C$ for all $\phi\in A$. 
\end{defi}
Now assume $X$ is Fano, and set $L:=-K_X$.
Given a K\"ahler form $\a\in c_1(X)$, consider Aubin's continuity method
\begin{equation}\label{e107}
  \Ric(\om_t)=t\om_t+(1-t)\a.
\end{equation}
It is well-known that there exists a unique maximal solution
$(\om_t)_{t\in[0,T)}$, where $0<T\le1$.
The following important result, due to Sz\'ekelyhidi~\cite{Sze3}, confirms a conjecture of Tian.

\begin{thm}\label{thm:sze}
  The set $A:=\left\{\om_t\mid t\in[0,T)\right\}$ 
  satisfies partial $C^0$-estimates at level $m$, for arbitrarily large positive integers $m$. 
\end{thm}

\corr{Given this result, we shall prove
\begin{thm}\label{thm:Kstable}
  Any uniformly K-stable Fano manifold admits a K\"ahler-Einstein metric. 
\end{thm}}
By working (much) harder, Datar and Sz\'ekelyhidi~\cite{DSz15} have in fact been able to deduce from Theorem~\ref{thm:sze} a much better result dealing with K-polystability and allowing a compact group action. 
%
%
\subsection{CM-stability and partial $C^0$-estimates}
We first present in some detail 
well-known ideas due to Tian~\cite[\S4.3]{Tian12}.
In this section, $(X,L)$ is an arbitrary polarized manifold.
\begin{prop}\label{prop:transfer} 
  Assume that $(X,mL)$ is CM-stable, and that $A\subset\cH$ 
  satisfies partial $C^0$-estimates at level $m$. 
  Then there exist $\d,C>0$ such that $M\ge\d J-C$ on $A$. 
\end{prop}
The proof, which is similar to the arguments in~\cite[\S5]{Sze3}. is based on two lemmas.
\begin{lem}\label{lem:shift} For any two metrics $\phi,\psi\in\cH$, we have
\begin{itemize}
\item[(i)] $|J(\phi)-J(\psi)|\le 2\sup(\phi-\psi)$; 

\item[(ii)] $M(\phi)\ge M(\psi)-C\sup|\phi-\psi|$ for some $C>0$ only depending on a one-sided bound (either upper or lower) for the Ricci curvature of the K\"ahler metric $dd^c\psi$. 
\end{itemize}
\end{lem}
\begin{proof} 
  Recall that 
  \begin{equation*}
    E(\phi)-E(\psi)
    =\frac{1}{n+1}\sum_{j=0}^nV^{-1}\int_X(\phi-\psi)(dd^c\phi)^j\wedge(dd^c\psi)^{n-j}.
  \end{equation*}
  As a consequence, $|E(\phi)-E(\psi)|\le\sup|\phi-\psi|$, and~(i) follows immediately. 
  
  For (ii), we basically argue as in the proof of~\cite[Lemma~3.1]{Tian14}. 
  By the Chen-Tian formula~\ref{equ:varM}, we have
  \begin{equation*}
    M(\phi)-M(\psi)
    =H_\psi(\phi)+\bar S\left(E(\phi)-E(\psi)\right)
    +E_{\Ric(dd^c\psi)}(\psi)-E_{\Ric(dd^c\psi)}(\phi). 
  \end{equation*}
  Here the entropy term $H_\psi(\phi)$ is non-negative, and we have 
  \begin{equation*}
    E_{\Ric(dd^c\psi)}(\phi)-E_{\Ric(dd^c\psi)}(\psi)
    =\sum_{j=0}^{n-1}V^{-1}\int_X(\phi-\psi)(dd^c\phi)^j\wedge(dd^c\psi)^{n-j-1}
    \wedge\Ric(dd^c\psi).
  \end{equation*}
  Assume $\Ric(dd^c\psi)\le C dd^c\psi$ for some constant $C>0$. 
  We may then write 
  \begin{multline*}
    (dd^c\phi)^j\wedge(dd^c\psi)^{n-j-1}\wedge\Ric(dd^c\psi)\\
    =C(dd^c\phi)^j\wedge(dd^c\psi)^{n-j}
    -(dd^c\phi)^j\wedge(dd^c\psi)^{n-j-1}\wedge(C'dd^c\psi-\Ric(dd^c\psi)), 
  \end{multline*}
  a difference of two positive measures of mass $CV$ and
  $CV+(L^{n-1}\cdot K_X)$, respectively, and the desired estimate follows. 
  
  The case where $\Ric(dd^c\psi)\ge-C' dd^c\psi$ is treated similarly 
  (and will anyway not be used in what follows). 
\end{proof}
We next recall a well-known upper bound for the Ricci curvature of restrictions of Fubini-Study metrics. 
\begin{lem}\label{lem:ric} 
  We have $\Ric(dd^c\phi)\le N_m dd^c\phi$ for all $\phi\in\cH_m$.  
\end{lem}
\begin{proof} 
  Choose a basis of $H^0(X,mL)$, and let $\om$ be the
  corresponding Fubini-Study metric on $\P:=\P H^0(X,mL)^*$. 
  Its curvature tensor 
  $$
  \Theta(T_\P,\om)\in C^\infty(\P,\La^{1,1}T_{\P}^*\otimes\mathrm{End}(T_\P))
  $$ 
  is Griffiths positive and satisfies
  $$
  \Tr_{T_\P}\Theta(T_\P,\om)=\Ric(\om)=N_m\om.
  $$
  
  For each complex submanifold $Y\subset\P$, the curvature of its tangent bundle $T_Y$ with respect to $\om|_Y$ satisfies $\Theta(T_Y,\om|_Y)\le\Theta(T_\P,\om)|_{T_Y}$ as $(1,1)$-forms on $Y$ with values in the endomorphisms of $T_Y$, as a consequence of a well-known curvature monotonicity property going back to Griffiths. We thus have
  $$
  \Ric(\om|_Y)=\Tr_{T_Y}\Theta(T_Y,\om|_Y)\le\Tr_{T_Y}\Theta(T_\P,\om)|_{T_Y}. 
  $$
  Using now $\Theta(T_\P,\om)\ge 0$, we have on the other hand
  $$
  \Tr_{T_Y}\Theta(T_\P,\om)|_{T_Y}\le\Tr_{T_\P}\Theta(T_\P,\om)|Y=N_m\om|_Y, 
  $$
  and hence 
  $$
  \Ric(\om|_Y)\le N_m\om|_Y.
  $$
  Applying this to the images of $X\subset\P$ under projective transformations yields the desired result. 
\end{proof}
\begin{proof}[Proof of Proposition~\ref{prop:transfer}] 
  Since $(X,mL)$ is CM-stable, there exist $\d,C>0$ such that 
  \begin{equation}
    M(P_m(\phi))\ge\d J(P_m(\phi))-C
  \end{equation}
  for all $\phi\in\cH$. By assumption on $A$, we also have 
  $|P_m(\phi)-\phi|\le C$ for all $\phi\in A$,
  and by Lemma~\ref{lem:ric}, the Ricci curvature of $dd^cP_m(\phi)$ is uniformly bounded 
  above. Hence Lemma~\ref{lem:shift} shows, 
  as desired, that there exists $C'>0$ with $M(\phi)\ge\d J(\phi)-C'$ 
  for all $\phi\in A$. 
\end{proof} 
%
%
\subsection{Proof of Theorem~\ref{thm:Kstable}}
Assume now that $X$ is a Fano manifold and set $L:=-K_X$.
Consider the continuity method~\eqref{e107}.
Pick metrics $\psi$ and $\phi_t$ on $-K_X$ 
such that $\a=dd^c\psi$ and $\om_t=dd^c\phi_t$, respectively.
After adding a constant to $\phi_t$,~\eqref{e107} may be written 
\begin{equation}\label{equ:Aubin}
  (dd^c\phi_t)^n=e^{-2\left(t\phi_t+(1-t)\psi\right)}.
\end{equation}
We recall the proof of the following well-known monotonicity property. 
\begin{lem}\label{lem:mono} 
  The function $t\to M(\phi_t)$ is non-increasing.
\end{lem}
\begin{proof} 
  We have 
  \begin{align*}
    -\frac{d}{dt}M(\phi_t)
    &=nV^{-1}\int_X\dot\phi_t\left(\Ric(\om_t)\wedge\om_t^{n-1}-\om_t^n\right)\\
    &=nV^{-1}(1-t)\int_X\dot\phi_t dd^c(\psi-\phi_t)\wedge(dd^c\phi_t)^{n-1}\\
    &=nV^{-1}(1-t)\int_X(\psi-\phi_t)dd^c\dot\phi_t\wedge(dd^c\phi_t)^{n-1}.
  \end{align*}
  Since $d^c$ is normalized so that $dd^c=\frac{i}{\pi}\partial\overline{\partial}$, we have 
 $$
 n\frac{dd^c\dot\phi_t\wedge\om_t^{n-1}}{\om_t^n}=\mathrm{tr}_{\om_t}dd^c\dot\phi_t=-\tfrac{1}{2\pi}\D''_t\dot\phi_t
 $$
 with $\D''_t$ denoting the $\bar\partial$-Laplacian with respect to $\om_t$. On the other hand, differentiating~\ref{equ:Aubin} yields
  \begin{equation*}
    ndd^c\dot\phi_t\wedge\om_t^{n-1}=2(\psi-\phi_t-t\dot\phi_t)\om_t^n, 
  \end{equation*}
 and hence 
 $$
 \psi -\phi_t =\left(t-\tfrac{1}{\pi}\Delta''_t\right) \dot\phi_t. 
 $$
  We get
  \begin{align*}
    -\frac{d}{dt}M(\phi_t)
    &=\frac{1-t}{2\pi}\int_X\left(\left(\tfrac{1}{\pi}\Delta''_t-t\right)\dot\phi_t\right)\left(\Delta''_t\dot\phi_t\right)\MA(\phi_t)\\
    &= \frac{1-t}{2\pi}\int_X\langle\left(\tfrac{1}{\pi}\Delta''_t-t\right)\bar\partial\dot\phi_t, \bar\partial\dot\phi_t\rangle_{\omega_t}\MA(\phi_t). 
  \end{align*}
  Since $\Ric(\om_t)\ge t\om_t$, the $\bar\partial$-Laplacian $\D''_t$
  satisfies $\tfrac{1}{\pi}\D''_t\ge t$ on $(0,1)$-forms, and the last
  integral is thus nonnegative. Indeed, this follows from the Bochner-Kodaira-Nakano identity applied to 
  $$
  C^\infty(X,\La^{0,1} T_X^*)\simeq C^\infty(X,\La^{n,1} T_X^*\otimes K_X^*)
  $$ 
  with the fiber metric $\p_t=-\half\log\omega_t^n$ on $K_X^*=-K_X$, with curvature $dd^c\p_t=\Ric(\om_t)$. 
\end{proof}
\corr{We may now complete the proof of Theorem~\ref{thm:Kstable}.
By Corollary~E, $(X,-mK_X)$ is CM-stable for all $m$ divisible enough. }
Theorem~\ref{thm:sze} and Proposition~\ref{prop:transfer} therefore
yield $\d,C>0$ such that $M(\phi_t)\ge\d J(\phi_t)-C$ along Aubin's
continuity method. Since $M(\phi_t)$ is bounded above by
Lemma~\ref{lem:mono}, it follows that $J(\phi_t)$ remains
bounded. By~\cite[Lemma 6.19]{TianBook}, the oscillation of $\phi_t$ is bounded, and well-known arguments allow us to conclude,
see~\cite[\S6.2]{TianBook}. 
%
%
%
%


\begin{thebibliography}{BBEGZ11}

\bibitem[Ant]{Ant} 
S.~Antonakoudis,
\newblock\emph{Valuative criteria of separatedness and properness}.
\newblock \texttt{https://www.dpmms.cam.ac.uk/~sa443/papers/criteria.pdf}

\bibitem[Berk90]{BerkBook}
V.~Berkovich.
\newblock \emph{Spectral theory and analytic geometry over non-Archimedean fields}.
\newblock Mathematical Surveys and Monographs, vol. 33.
\newblock American Mathematical Society, Providence, RI, 1990.

\bibitem[Berk09]{BerkHodge}
V.~G.~Berkovich.
\newblock \emph{A non-Archimedean interpretation of the weight zero
  subspaces of limit mixed Hodge structures}.
\newblock In \emph{Algebra, arithmetic, and geometry: in honor of Yu.~I.~Manin}.
\newblock Progr. Math., vol 269, 49--67.
\newblock Birkh\"auser, Boston, MA, 2009.

\bibitem[Berm16]{Berm16}
R.~J.~Berman.
\newblock \emph{K-polystability of $\Q$-Fano varieties admitting K\"ahler-Einstein metrics}.
\newblock Invent. Math. \textbf{203} (2016), 973--1025.

\bibitem[BB14]{BB}
R.~J.~Berman and B.~Berndtsson. 
\newblock\emph{Convexity of the K-energy on the space of K\"{a}hler metrics}.
\newblock \texttt{arXiv:1405.0401}. 

\bibitem[BBEGZ11]{BBEGZ} 
R.~J.~Berman, S.~Boucksom, P.~Eyssidieux, V.~Guedj and A.~Zeriahi. 
\newblock \emph{K\"ahler--Einstein metrics and the K\"ahler--Ricci flow on log Fano varieties}.
\newblock \texttt{arXiv:1111.7158}. 

\bibitem[BBGZ13]{BBGZ} 
R.~J.~Berman, S.~Boucksom, V.~Guedj and A.~Zeriahi. 
\newblock \emph{A variational approach to complex Monge-Amp\`ere equations}.
\newblock  Publ. Math. Inst. Hautes \'Etudes Sci. \textbf{117} (2013), 179--245. 
   
\bibitem[BBJ15]{BBJ15}
R.~J.~Berman, S.~Boucksom and M.~Jonsson.
\newblock\emph{A variational approach to the Yau-Tian-Donaldson conjecture}.
\newblock \texttt{arXiv:1509.04561}.

\bibitem[BDL16]{BDL16}
R.~J.~Berman, T.~Darvas and C.~H.~Lu.
\newblock \emph{Regularity of weak minimizers of the K-energy and
applications to properness and K-stability}.
\newblock \texttt{arXiv:1602.03114}.

\bibitem[BFJ16]{siminag}
S.~Boucksom, C.~Favre and M.~Jonsson.
\newblock \emph{Singular semipositive metrics in non-Archimedean geometry}.
\newblock J. Algebraic Geom. \textbf{25} (2016), 77--139. 

\bibitem[BFJ15a]{nama}
S.~Boucksom, C.~Favre and M.~Jonsson.
\newblock \emph{Solution to a non-Archimedean Monge-Amp\`ere equation}.
\newblock J. Amer. Math. Soc., \textbf{28} (2015), 617--667.

\bibitem[BHJ15]{BHJ1}
S.~Boucksom, T.~Hisamoto and M.~Jonsson. 
\newblock \emph{Uniform K-stability, Duistermaat-Heckman measures and singularities of pairs}.
\newblock \texttt{arXiv:1504.06568}.
\newblock To appear in Ann. Inst. Fourier. 

\bibitem[BJ16a]{konsoib}
S.~Boucksom and M.~Jonsson. 
\newblock \emph{Tropical and non-Archimedean limits of degenerating families of volume forms}.
\newblock \texttt{arXiv:1605.05277}.
\newblock To appear in J. {\'E}c. polytech. Math.
     
\bibitem[BJ16b]{trivval}
S.~Boucksom and M.~Jonsson. 
\newblock \emph{Singular semipositive metrics on line bundles on
  varieties over trivially valued fields}.
\newblock In preparation. 

\bibitem[Che00]{Che2} 
X.X.~Chen. 
\newblock\emph{On the lower bound of the Mabuchi energy and its application}. 
\newblock Int. Math. Res. Not. (2000), no. 12, 607--623. 
   
\bibitem[CDS15]{CDS15} 
X.X.~Chen, S.~K.~Donaldson and S.~Sun. 
\newblock\emph{K\"ahler-Einstein metrics on Fano manifolds, I-III}.
\newblock J. Amer. Math. Soc. \textbf{28} (2015), 
  183--197, 199--234, 235--278.

\bibitem[CSW15]{CSW15} 
X.X.~Chen, S.~Sun and B.~Wang.
\newblock\emph{K\"ahler-Ricci flow, K\"ahler-Einstein metric, and K-stability}.
\newblock \texttt{arXiv:1508.04397}.

\bibitem[DR15]{DR} 
T.~Darvas and Y.~A.~Rubinstein
\newblock\emph{Tian's properness conjectures and Finsler geometry of the space of K\"ahler metrics}.
\newblock \texttt{arXiv:1506.07129}.  
\newblock To appear in J. Amer. Math. Soc.

\bibitem[DSz15]{DSz15} 
V.~Datar and G.~Sz\'ekelyhidi.
\newblock\emph{K\"ahler-Einstein metrics along the smooth continuity method}. 
\newblock Geom. Funct. Anal. \textbf{26} (2016), 975–-1010.
    
\bibitem[Der15]{Der1} 
R.~Dervan. 
\newblock\emph{Uniform stability of twisted constant scalar curvature K\"ahler metrics}. 
\newblock Int. Math. Res. Not. (2016), no. 15, 4728--4783.

\bibitem[DR16]{DR16} 
R.~Dervan and J.~Ross.
\newblock\emph{K-stability for K\"ahler manifolds}.
\newblock \texttt{arXiv:1602.08983}.

\bibitem[Din88]{Din88}
W.-Y. Ding.
\newblock \emph{Remarks on the existence problem for positive
  K\"ahler-Einstein metrics}.
\newblock Math. Ann. \textbf{282} (1988) 463--471.

\bibitem[DT92]{DT92}
W.-Y. Ding and G.~Tian.
\newblock \emph{K\"ahler-Einstein metrics and the generalized Futaki invariant}.
\newblock Invent. Math. \textbf{110} (1992), 315--335.

\bibitem[Don99]{Don99}
S.~K.~Donaldson. 
\newblock \emph{Symmetric spaces, K\"ahler geometry and Hamiltonian
  dynamics}. 
\newblock Northern California Symplectic Geometry Seminar, 13--33.
\newblock Amer. Math. Soc. Transl. Ser. 2, 196.
\newblock Amer. Math. Soc., Providence, RI, 1999. 

\bibitem[Don02]{Don2}
 S.~K.~Donaldson. 
\newblock \emph{Scalar curvature and stability of toric varieties}.
\newblock  J. Differential Geom. \textbf{62} (2002), 289--349.  
 
\bibitem[Don05]{Don3}
 S.~K.~Donaldson. 
\newblock \emph{Lower bounds on the Calabi functional}.
\newblock  J. Differential Geom. \textbf{70} (2005), 453--472.

\bibitem[Don12]{Don10}
S.K.~Donaldson.
\newblock\emph{Stability, birational transformations and the Kahler-Einstein problem}. 
\newblock Surv. Diff. Geom., vol. 17, International Press, Boston, MA (2012), 203--228.

\bibitem[Elk89]{Elk1} 
R.~Elkik.
\newblock \emph{Fibr\'es d'intersections et int\'egrales de classes de Chern}.  
\newblock Ann. Sci. \'Ecole Norm. Sup. (4) \textbf{22} (1989), 195--226. 
      
\bibitem[Elk90]{Elk2} 
R.~Elkik.
\newblock \emph{M\'etriques sur les fibr\'es d'intersection}. 
\newblock Duke Math. J. \textbf{61} (1990), 303--328. 
       
\bibitem[Fuj15]{Fuj15} 
K.~Fujita.
\newblock \emph{Optimal bounds for the volumes of K\"ahler-Einstein
  Fano manifolds}.
\newblock \texttt{arXiv:1508.04578}.

\bibitem[Fuj16]{Fuj16} 
K.~Fujita.
\newblock \emph{A valuative criterion for uniform K-stability of
  $\mathbb{Q}$-Fano varities}.
\newblock \texttt{arXiv:1602.00901v1}.

\bibitem[His16]{His16}
T. Hisamoto.
\newblock \emph{On the limit of spectral measures associated to a test configuration of a polarized K\"{a}hler manifold}. 
\newblock J. Reine Angew. Math. \textbf{713} (2016), 129--148.

\bibitem[Jon16]{amoebae}
M.~Jonsson.
\newblock\emph{Degenerations of amoebae and Berkovich spaces}.
\newblock Math. Ann. \textbf{364} (2016), 293--311.

\bibitem[Kap13]{Kap} H.M.~Kapadia. 
\newblock\emph{Deligne pairings and discriminants of algebraic varieties}.
\newblock \texttt{arXiv:1312.7870}. 

\bibitem[Kem]{Kem} G.~Kempf. 
\newblock\emph{Algebraic varieties}.
\newblock LMS Lecture Note Series {\bf 172}. 
\newblock Cambridge University Press, Cambridge, 1993. 


\bibitem[Li12]{Li12} 
C.~Li. 
\newblock  \emph{K\"ahler-Einstein metrics and K-stability}.
\newblock Ph.D. Thesis, Princeton University, 2012.
   
\bibitem[LX14]{LX}
C.~Li and C.~Xu.
\newblock  \emph{Special test configurations and K-stability of Fano varieties}.
\newblock Ann.\ of Math. \textbf{180} (2014), 197--232.

\bibitem[Li14]{Li14} 
L.~Li. 
\newblock \emph{Subharmonicity of conic Mabuchi's functional, I}.
\newblock \texttt{arXiv:1511.00178}.

\bibitem[Mab87]{Mab87}
T.~Mabuchi.
\newblock\emph{Some symplectic geometry on compact K\"ahler manifolds. I}.
\newblock Osaka J. Math. \textbf{24} (1987), 227-–252.


\bibitem[Mor99]{Mor} 
A.~Moriwaki. 
\newblock\emph{The continuity of Deligne's pairing}. 
\newblock Internat. Math. Res. Notices \textbf{19} (1999) 1057--1066. 

\bibitem[MFF]{GIT} 
D.~Mumford, J.~Fogarty and F.~Kirwan.
\newblock\emph{Geometric invariant theory}. 
\newblock Third edition. 
\newblock Ergebnisse der Mathematik und ihrer Grenzgebiete (2), 34.
\newblock Springer-Verlag, Berlin, 1994.

\bibitem[MG00]{MG} 
E.~Mu\~{n}oz Garcia.
\newblock\emph{Fibr\'es d'intersection}. 
\newblock Compositio Math. \textbf{124} (2000) 219--252. 

\bibitem[Oda13]{Oda13}
Y. Odaka.
\newblock\emph{A generalization of the Ross-Thomas slope theory}.
\newblock Osaka J. Math. \textbf{50} (2013), 171--185. 

\bibitem[Pau04]{Paul04} 
S.~T.~Paul. 
\newblock \emph{Geometric analysis of Chow Mumford stability}.
\newblock Adv. Math. \textbf{182} (2004), 333--356.

\bibitem[Pau12]{Paul12} 
S.~T.~Paul.
\newblock \emph{Hyperdiscriminant polytopes, Chow polytopes and  Mabuchi energy asymptotics}.
\newblock Ann.\ of Math. (2) \textbf{175} (2012), 255--296.

\bibitem[Pau13]{Paul13}
S.~T.~Paul.
\newblock \emph{Stable pairs and coercive estimates for the Mabuchi functional}.
\newblock \texttt{arXiv:1308.4377}.

\bibitem[PT06]{PT1} 
S. T. Paul and G. Tian.
\newblock \emph{CM Stability and the Generalized Futaki Invariant I}.
\newblock \texttt{arXiv:math/0605278}.

\bibitem[PT09]{PT2} 
S.~T.~Paul and G. Tian.
\newblock \emph{CM Stability and the Generalized Futaki Invariant II}.
\newblock Ast\'{e}risque No. \textbf{328} (2009), 339--354.

\bibitem[PRS08]{PRS} 
D.~H. Phong, J.~Ross, and J.~Sturm.
\newblock \emph{Deligne pairings and the Knudsen-Mumford expansion}.
\newblock J. Differential Geom. \textbf{78} (2008), 475--496.
  
\bibitem[PSSW08]{PSSW} 
D.~H.~Phong, J.~Song, J.~Sturm and B.~Weinkove.
\newblock \emph{The Moser-Trudinger inequality on K{\"a}hler-Einstein manifolds}. 
\newblock Amer. J. Math.  \textbf{130}  (2008),  1067--1085.
  
\bibitem[PS04]{PS6}
D.~H.~Phong and J.~Sturm.
\newblock \emph{Scalar curvature, moment maps, and the Deligne pairing}. 
\newblock Amer. J. Math. \textbf{126} (2004), no. 3, 693--712. 
   
\bibitem[Sem92]{Sem92}
S.~Semmes.
\newblock\emph{Complex Monge-Amp\`ere and symplectic manifolds}. 
\newblock Amer. J. Math. \textbf{114} (1992), 495--550.

\bibitem[SD16]{SD16}
Z.~Sj\"ostr\"om Dyrefelt.
\newblock{K-semistability of cscK manifolds with transcendental
  cohomology class}.
\newblock \texttt{arXiv:1601.07659}.

\bibitem[Stol66]{Stol} W.~Stoll. 
\newblock\emph{The continuity of the fiber integral}. 
\newblock Math. Z. \textbf{95} (1966), no.2, 87--138. 

\bibitem[Stop09]{Sto09}
J.~Stoppa. 
\newblock \emph{K-stability of constant scalar curvature K\"ahler manifolds}. 
\newblock Adv. Math. \textbf{221} (2009) 1397--1408.

\bibitem[Sz\'e06]{Sze1}
G. Sz\'{e}kelyhidi. 
\newblock  \emph{Extremal metrics and K-stability}.
\newblock Ph.D Thesis. \texttt{arXiv:math/0611002}. 

\bibitem[Sz\'e15]{Sze2}
G.~Sz\'ekelyhidi. 
\newblock \emph{Filtrations and test-configurations}. 
\newblock With an appendix by S.~Boucksom.
\newblock Math. Ann. \textbf{362} (2015), 451--484.

\bibitem[Sz\'e16]{Sze3}
G.~Sz\'ekelyhidi. 
\newblock \emph{The partial $C^0$-estimate along the continuity method}.
\newblock J. Amer. Math. Soc. \textbf{29} (2016), 537--560.
   
\bibitem[Tho06]{Tho} 
R.~Thomas. 
\newblock \emph{Notes on GIT and symplectic reduction for bundles and
  varieties}.
\newblock Surveys in Differential Geometry, Vol. 10, 221--273.
\newblock Int. Press, Somerville, MA, 2006.

\bibitem[Tia97]{Tian97} 
G.~Tian.
\newblock \emph{K{\"a}hler-Einstein metrics with positive scalar curvature}.
\newblock Inv. Math. \textbf{130} (1997), 239--265.

\bibitem[Tia00]{TianBook} 
G.~Tian.
\newblock \emph{Canonical metrics in K\"ahler geometry}.
\newblock Notes taken by Meike Akveld. 
\newblock Lectures in Mathematics ETH Z\"urich. 
\newblock Birkh\"auser Verlag, Basel, 2000. 
 
\bibitem[Tia12]{Tian12} 
G.~Tian.
\newblock \emph{Existence of Einstein metrics on Fano manifolds}. 
\newblock Metric and differential geometry, 119--159,
\newblock Progr. Math., \textbf{297}, Birkh\"auser/Springer, Basel, 2012. 
    
\bibitem[Tia17]{Tian14} 
G.~Tian.
\newblock \emph{K-stability implies CM-stability}.
\newblock In: Bost J.-B., Hofer H., Labourie F., Le Jan Y., Ma X., Zhang W. (eds) Geometry, Analysis and Probability. \newblock Progress in Mathematics {\bf 310}. Birkh\"auser, Cham. 

\bibitem[Tia15]{Tian15} 
G.~Tian.
\newblock \emph{K-stability and K\"ahler-Einstein metrics}.
\newblock Comm. Pure Appl. Math. \textbf{68} (2015), 1085--1156. 
    
\bibitem[Wan12]{Wan12} 
X.-Wang. 
\newblock\emph{Height and GIT weight}.
\newblock Math. Res. Lett. \textbf{19} (2012), 909--926.

\bibitem[Zha96]{Zha96} 
S.-W.~Zhang. 
\newblock\emph{Heights and reductions of semi-stable varieties}.  
\newblock Compositio Math. \textbf{104} (1996), 77--105.

\end{thebibliography}
\end{document}